\documentclass[a4paper,12pt]{article}
\usepackage{a4wide}
\usepackage{graphicx,eepic, bm, times}
\usepackage{amsthm,amsmath,amsfonts,amssymb}

\title{Generalised golden ratios over integer alphabets}
\author{Simon Baker}

\newtheorem{thm}{Theorem}[section]
\newtheorem{cor}[thm]{Corollary}
\newtheorem{lem}[thm]{Lemma}

\newtheorem{prop}[thm]{Proposition}

\newtheorem{remark}[thm]{Remark}

\begin{document}
\maketitle

\begin{abstract}
It is a well known result that for $\beta\in(1,\frac{1+\sqrt{5}}{2})$ and $x\in(0,\frac{1}{\beta-1})$ there exists uncountably many $(\epsilon_{i})_{i=1}^{\infty}\in \{0,1\}^{\mathbb{N}}$ such that $x=\sum_{i=1}^{\infty}\epsilon_{i}\beta^{-i}.$ When $\beta\in(\frac{1+\sqrt{5}}{2},2]$ there exists $x\in (0,\frac{1}{\beta-1})$ for which there exists a unique $(\epsilon_{i})_{i=1}^{\infty}\in \{0,1\}^{\mathbb{N}}$ such that $x=\sum_{i=1}^{\infty}\epsilon_{i}\beta^{-i}.$ In this paper we consider the more general case when our sequences are elements of $\{0,\ldots,m\}^{\mathbb{N}}.$ We show that an analogue of the golden ratio exists and give an explicit formula for it.
\end{abstract}

\section{Introduction}
\let\thefootnote\relax\footnote{AMS Classification: 37A45, 37C45} \let\thefootnote\relax\footnote{Keywords: Beta-expansions, Dimension theory} 
Let $m\in\mathbb{N},$ $\beta\in(1,m+1]$ and $I_{\beta,m}=[0,\frac{m}{\beta-1}]$. Each $x\in I_{\beta,m}$ has an expansion of the form $$x=\sum_{i=1}^{\infty}\frac{\epsilon_{i}}{\beta^{i}},$$ for some $(\epsilon_{i})_{i=1}^{\infty}\in \{0,\ldots,m\}^{\mathbb{N}}.$ We call such a sequence a \textit{$\beta$-expansion} for $x$. For $x\in I_{\beta,m}$ we denote the set of $\beta$-expansions for $x$ by $\Sigma_{\beta,m}(x)$, i.e., $$\Sigma_{\beta,m}(x)=\Big\{(\epsilon_{i})_{i=1}^{\infty}\in \{0,\ldots,m\}^{\mathbb{N}} : \sum_{i=1}^{\infty}\frac{\epsilon_{i}}{\beta^{i}}=x\Big\}.$$ In \cite{Erdos} the authors consider the case when $m=1$, they show that for $\beta\in(1,\frac{1+\sqrt{5}}{2})$ the set $\Sigma_{\beta,1}(x)$ is uncountable for every $x\in (0,\frac{1}{\beta-1})$. The endpoints of $[0,\frac{1}{\beta-1}]$ trivially have a unique $\beta$-expansion. In \cite{DaKa} it is shown that for $\beta\in(\frac{1+\sqrt{5}}{2},2]$ there exists $x\in (0,\frac{1}{\beta-1})$ with a unique $\beta$-expansion.

For $m\in\mathbb{N}$ we define $\mathcal{G}(m)\in\mathbb{R}$ to be a \textit{generalised golden ratio for $m$} if for $\beta\in(1,\mathcal{G}(m))$ the set $\Sigma_{\beta,m}(x)$ is uncountable for every $x\in(0,\frac{m}{\beta-1})$, and for $\beta\in(\mathcal{G}(m),m+1]$ there exists $x\in(0,\frac{m}{\beta-1})$ for which $\left|\Sigma_{\beta,m}(x)\right|=1.$

In \cite{KomLaiPed} the authors consider a similar setup. They consider the case where $\beta$-expansions are elements of $\{a_{1},a_{2},a_{3}\}^{\mathbb{N}},$ for some $a_{1},a_{2},a_{3}\in\mathbb{R}$. They show that for each ternary alphabet there exists a constant $G\in\mathbb{R}$ such that, there exists nontrivial unique $\beta$-expansions if and only if $\beta>G.$ Moreover they give an explicit formula for $G$. 

Our main result is the following.

\begin{thm}
\label{Main thm}
For each $m\in\mathbb{N}$ a generalised golden ratio exists and is equal to:
\begin{equation}
\label{Golden ratio}
\mathcal{G}(m) = \left\{ \begin{array}{rl}
k+1 &\mbox{ if $m=2k$} \\
\frac{k+1+\sqrt{k^{2}+6k+5}}{2}   &\mbox{ if $m=2k+1$.}
       \end{array} \right.
\end{equation}
\end{thm}
\begin{remark}
$\mathcal{G}(m)$ is a Pisot number for all $m\in\mathbb{N}$. Recall a Pisot number is a real algebraic integer greater than $1$ whose Galois conjugates are of modulus strictly less than $1$.
\end{remark}In section 6 we include a table of values for $\mathcal{G}(m).$ We prove Theorem \ref{Main thm} in section $3$. In section $4$ we consider the set of points with unique $\beta$-expansion for $\beta\in(\mathcal{G}(m),m+1],$ and in section $5$ we study the growth rate and dimension theory of the set of $\beta$-expansions for $\beta\in(1,\mathcal{G}(m)).$

\section{Preliminaries}
Before proving Theorem \ref{Main thm} we require the following preliminary results and theory. Let $m\in\mathbb{N}$ be fixed and $\beta\in(1,m+1]$. For $i\in\{0,\ldots,m\}$ we fix $T_{\beta,i}(x)=\beta x -i.$ The proof of the following lemma is trivial and therefore omitted.
\begin{lem}
\label{Map properties}
The map $T_{\beta,i}$ satisfies the following:
\begin{itemize}
  \item $T_{\beta,i}$ has a unique fixed point equal to $\frac{i}{\beta-1}.$
	\item $T_{\beta,i}(x)> x$ for all $x>\frac{i}{\beta-1},$
	\item $T_{\beta,i}(x)< x$ for all $x<\frac{i}{\beta-1},$
	\item $|T_{\beta,i}(x)-T_{\beta,i}(\frac{i}{\beta-1})|=\beta|x-\frac{i}{\beta-1}|$,  for all $x\in\mathbb{R},$ that is $T_{\beta,i}$ scales the distance between the fixed point $\frac{i}{\beta-1}$ and an arbitrary point by a factor $\beta$. 
\end{itemize}
\end{lem}Understanding where in $I_{\beta,m}$ these fixed points are will be important in our later analysis. 

We let
\begin{align*}
\Omega_{\beta,m}(x)=\Big\{&(a_{i})_{i=1}^{\infty}\in \{T_{\beta,0}\ldots T_{\beta,m}\}^{\mathbb{N}}:(a_{n}\circ a_{n-1}\circ \ldots \circ a_{1})(x)\in I_{\beta,m}\\
& \textrm{ for all } n\in\mathbb{N}\Big\}.
\end{align*} Similarly we define $$\Omega_{\beta,m,n}(x)=\Big\{(a_{i})_{i=1}^{n}\in \{T_{\beta,0}\ldots T_{\beta,m}\}^{n}:(a_{n}\circ a_{n-1}\circ \ldots \circ a_{1})(x)\in I_{\beta,m} \Big\}.$$
Typically we will denote an element of $\Omega_{\beta,m,n}(x)$ or any finite sequence of maps by $a$. When we want to emphasise the length of $a$ we will use the notation $a^{(n)}$. We also adopt the notation $a^{(n)}(x)$ to mean $(a_{n}\circ a_{n-1}\circ \ldots \circ a_{1})(x).$
\begin{remark}
It is important to note that if for some finite sequence of maps $a,$ $a(x)\notin I_{\beta,m}$ then we cannot concatenate $a$ by any finite sequence of maps $b,$ such that $b(a(x))\in I_{\beta,m}.$
\end{remark}
\begin{remark}
Let $\beta\in(1,m+1],$ for any $x\in I_{\beta,m}$ there always exists $i\in\{0,\ldots,m\}$ such that $T_{\beta,i}(x)\in I_{\beta,m}.$ For $\beta>m+1$ such an $i$ does not always exist.
\end{remark}
\begin{lem}
\label{Bijection lemma}
$\left|\Sigma_{\beta,m}(x)\right|=\left|\Omega_{\beta,m}(x)\right|.$
\end{lem}
\begin{proof}
It is a simple exercise to show that $$\Sigma_{\beta,m}(x)=\Big\{(\epsilon_{i})_{i=1}^{\infty}\in \{0,\ldots,m\}^{\mathbb{N}}: x-\sum_{i=1}^{n}\frac{\epsilon_{i}}{\beta^{i}}\in\Big[0,\frac{m}{\beta^{n}(\beta-1)}\Big] \textrm{ for all } n\in\mathbb{N}\Big\}.$$ Following \cite{FengSid} we observe that
\begin{align*}
\Sigma_{\beta,m}(x)&=\Big\{(\epsilon_{i})_{i=1}^{\infty}\in \{0,\ldots,m\}^{\mathbb{N}}: x-\sum_{i=1}^{n}\frac{\epsilon_{i}}{\beta^{i}}\in\Big[0,\frac{m}{\beta^{n}(\beta-1)}\Big] \textrm{ for all } n\in\mathbb{N}\Big\}\\
&=\Big\{(\epsilon_{i})_{i=1}^{\infty}\in \{0,\ldots,m\}^{\mathbb{N}}: \beta^{n}x-\sum_{i=1}^{n}\epsilon_{i}\beta^{n-i}\in I_{\beta,m} \textrm{ for all } n\in\mathbb{N}\Big\}\\
&=\Big\{(\epsilon_{i})_{i=1}^{\infty}\in \{0,\ldots,m\}^{\mathbb{N}}:(T_{\beta,\epsilon_{n}}\circ \ldots \circ T_{\beta,\epsilon_{1}})(x)\in I_{\beta,m}\textrm{ for all } n\in\mathbb{N}\Big\}.
\end{align*}Our result follows immediately.
\end{proof}By Lemma \ref{Bijection lemma} we can rephrase the definition of a generalised golden ratio in terms of the set $\Omega_{\beta,m}(x).$ This equivalent definition will be more suitable for our purposes. The set $\Omega_{\beta,m,n}(x)$ will be useful when we study the growth rate and dimension theory of the set of $\beta$-expansions.

For a point $x\in I_{\beta,m}$ we can take $i$ to be the first digit in a $\beta$-expansion for $x$ if and only if $\beta x-i\in I_{\beta,m}.$ This is equivalent to 
$$x\in\Big[\frac{i}{\beta},\frac{i\beta +m-i}{\beta(\beta-1)}\Big],$$ as such we refer to the interval $[\frac{i}{\beta},\frac{i\beta +m-i}{\beta(\beta-1)}]$ as the \textit{$i$-th digit interval}. Generally speaking we can take $i$ to be the $j$-th digit in a $\beta$-expansion for $x$ if and only if there exists $a\in\Omega_{\beta,m,j-1}(x)$ such that, $a(x)\in [\frac{i}{\beta},\frac{i\beta +m-i}{\beta(\beta-1)}].$ When $x$ or an image of $x$ is contained in the intersection of two digit intervals we have a choice of digit in our $\beta$-expansion for $x$. Generally speaking any two digit intervals may intersect for $\beta$ sufficiently small, however for our purposes we need only consider the case when the $i$-th digit interval intersects the adjacent $(i-1)$-th or $(i+1)$-th digit intervals, for some $i\in\{0,\ldots,m\}$. Any intersection of this type is of the form 
$$\Big[\frac{i}{\beta},\frac{(i-1)\beta+m-(i-1)}{\beta(\beta-1)}\Big],$$for some $i\in\{1,\ldots,m\}.$ In what follows we refer to the interval $[\frac{i}{\beta},\frac{(i-1)\beta+m-(i-1)}{\beta(\beta-1)}]$ as the \textit{$i$-th choice interval}. Both $T_{\beta,i-1}$ and $T_{\beta,i}$ map the $i$-th choice interval into $I_{\beta,m}$. These intervals always exist and are nontrivial for $\beta\in(1,m+1).$ 

\begin{prop}
\label{Uncountable proposition}
Suppose for any $x\in(0,\frac{m}{\beta-1})$ there always exists a finite sequence of maps that map $x$ into the interior of a choice interval, then $\Omega_{\beta,m}(x)$ is uncountable.
\end{prop}
The proof of this proposition is essentially contained in the proof of Theorem 1 in \cite{SidQueenMary}.
\begin{proof}
Let $x\in(0,\frac{m}{\beta-1})$. Suppose there exists $n\in\mathbb{N}$ and $a\in\Omega_{\beta,m,n}(x)$ such that $a(x)\in(\frac{i}{\beta},\frac{(i-1)\beta+m-(i-1)}{\beta(\beta-1)}),$ for some $i\in\{1,\ldots,m\}.$ As $a(x)$ is an element of the interior of a choice interval both $T_{\beta,i-1}(a(x))\in(0,\frac{m}{\beta-1})$ and $T_{\beta,i}(a(x))\in(0,\frac{m}{\beta-1}).$ As such our hypothesis applies to both $T_{\beta,i-1}(a(x))$ and $T_{\beta,i}(a(x)),$ and we can assert that there exists a finite sequence of maps that map these two distinct images of $x$ into the interior of another choice interval. Repeating this procedure arbitrarily many times it is clear that $\Omega_{\beta,m}(x)$ is uncountable.
\end{proof}
By Proposition \ref{Uncountable proposition}, to prove Theorem \ref{Main thm} it suffices to show that for $\beta\in(1,\mathcal{G}(m))$ every $x\in (0,\frac{m}{\beta-1})$ can be mapped into the interior of a choice interval, and for $\beta\in(\mathcal{G}(m),m+1]$ there exists $x\in(0,\frac{m}{\beta-1})$ that never maps into a choice interval.

We define the \textit{switch region} to be the interval $$\Big[\frac{1}{\beta},\frac{(m-1)\beta +1}{\beta(\beta-1)}\Big].$$ The significance of this interval is that if a point $x$ has a choice of digit in the $j$-th entry of a $\beta$-expansion, then there exists $a\in\Omega_{\beta,m,j-1}(x)$ such that $a(x)\in[\frac{1}{\beta},\frac{(m-1)\beta +1}{\beta(\beta-1)}]$. The following lemmas are useful in understanding the dynamics of the maps $T_{\beta,i}$ around the switch region, understanding these dynamics will be important in our proof of Theorem \ref{Main thm}.

\begin{lem}
\label{switch lemma}
For $\beta\in(1,\frac{m+\sqrt{m^{2}+4}}{2})$ and $x\in(0,\frac{m}{\beta-1})$ there exists a finite sequence of maps that map $x$ into the interior of our switch region. 
\end{lem}
\begin{proof} 
If $x$ is contained within the interior of the switch region we are done, let us suppose otherwise. By the monotonicity of the maps $T_{\beta,0}$ and $T_{\beta,m}$ it suffices to show that $$T_{\beta,0}\Big(\frac{1}{\beta}\Big)<\frac{(m-1)\beta+1}{\beta(\beta-1)}\textrm{ and } T_{\beta,m}\Big(\frac{(m-1)\beta+1}{\beta(\beta-1)}\Big)>\frac{1}{\beta}.$$ Both of these inequalities are equivalent to $\beta^{2}-m\beta-1<0,$ applying the quadratic formula we can conclude our result.
\end{proof}
\begin{remark}
When $m=1$ the switch region is a choice interval. An application of Lemma \ref{Bijection lemma}, Proposition \ref{Uncountable proposition} and Lemma \ref{switch lemma} yields the result stated in \cite{Erdos}, i.e, for $\beta\in(1,\frac{1+\sqrt{5}}{2})$ and $x\in(0,\frac{1}{\beta-1})$ the set $\Sigma_{\beta,1}(x)$ is uncountable.
\end{remark}
\begin{lem}
\label{interior choices}
For $\beta\in(1,\frac{m+2}{2})$ every $x$ in the interior of the switch region is contained in the interior of a choice interval.
\end{lem}
\begin{proof}
It suffices to show that for each $i\in\{1,2,\ldots,m-1\}$ the $(i-1)$-th and $(i+1)$-th digit intervals intersect in a nontrivial interval. This is equivalent to $$\frac{i+1}{\beta}<\frac{(i-1)\beta+m-(i-1)}{\beta(\beta-1)},$$ a simple manipulation yields that this is equivalent to $\beta<\frac{m+2}{2}.$
\end{proof} We refer the reader to Figure \ref{fig1} for a diagram depicting the case where $\beta<\frac{m+2}{2}.$ For $i\in\{1,2,\ldots,m-1\}$ and $\beta\geq \frac{m+2}{2}$ the interval $$\Big[\frac{(i-1)\beta+m-(i-1)}{\beta(\beta-1)},\frac{i+1}{\beta}\Big]$$ is well defined. We refer to this interval as the \textit{$i$-th fixed digit interval}. The significance of this interval is that if a point $x$ is contained in the interior of the $i$-th fixed digit interval only $T_{\beta,i}$ maps $x$ into $I_{\beta,m}$. Similarly we define the $0$-th fixed digit interval to be $[0,\frac{1}{\beta}]$ and the $m$-th fixed digit interval to be $[\frac{(m-1)\beta+1}{\beta(\beta-1)},\frac{m}{\beta-1}].$ Understanding how the different $T_{\beta,i}$'s behave on these intervals will be important when it comes to constructing generalised golden ratios in the case where $m$ is odd.

\begin{figure}[t]
\centering \unitlength=0.54mm
\begin{picture}(150,150)(0,-10)
\thinlines
\path(35,0)(0,0)(0,150)(150,150)(150,0)(115,0)
\put(-2,-7){$0$}
\put(143,-7){$\frac{m}{\beta-1}$}
\put(33,-8){$\frac{1}{\beta}$}
\put(101,-9){$\frac{(m-1)\beta+1}{\beta(\beta-1)}$}

\thicklines
\path(0,0)(80,150)
\path(35,0)(115,150)
\path(70,0)(150,150)
\path(70,0)(80,0)
\path(70,2)(70,-2)
\path(80,2)(80,-2)
\dottedline(80,150)(80,0)
\dottedline(115,150)(115,0)
\dottedline(35,0)(70,0)
\dottedline(80,0)(115,0)
\end{picture}
\caption{The case where $\beta\in(1,\frac{m+2}{2})$}
    \label{fig1}
\end{figure}
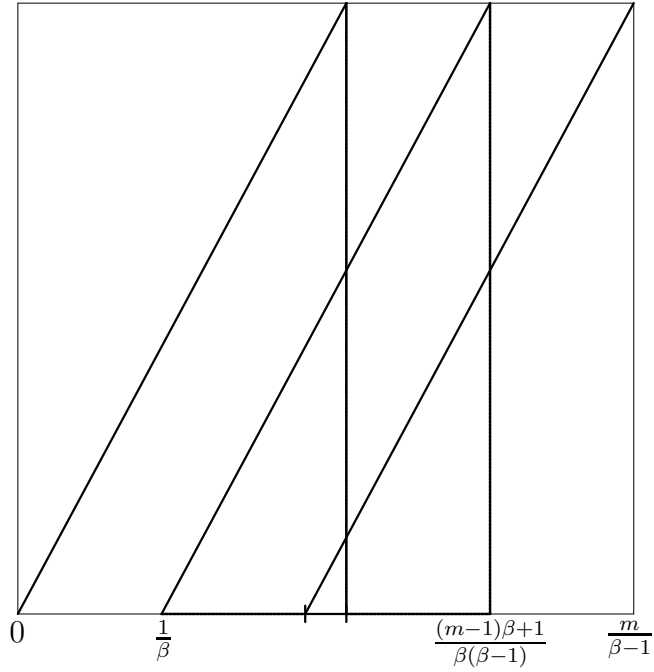

\section{Proof of Theorem \ref{Main thm}}
We are now in a position to prove Theorem \ref{Main thm}, for ease of exposition we reduce our analysis to two cases, when $m$ is even and when $m$ is odd. 

\subsection{Case where $m$ is even}
In what follows we assume $m=2k$ for some $k\in \mathbb{N}$.
\begin{prop}
\label{Even prop 1}
For $\beta\in(1,k+1)$ every $x\in(0,\frac{m}{\beta-1})$ has uncountably many $\beta$-expansions.
\end{prop}
\begin{proof}
By Lemma \ref{Bijection lemma} and Proposition \ref{Uncountable proposition} it suffices to show that every $x\in(0,\frac{m}{\beta-1})$ can be mapped into the interior of a choice interval. It is a simple exercise to show that $\frac{m+2}{2}<\frac{m^{2}+\sqrt{m^{2}+4}}{2}$ for all $m\in\mathbb{N},$ as such for $\beta\in(1,k+1)$ we can apply Lemma \ref{switch lemma}, therefore there exists a sequence of maps that map $x$ into the interior of the switch region. By Lemma \ref{interior choices} every point in the interior of our switch region is contained in the interior of a choice interval. 
\end{proof}

\begin{prop}
\label{Even prop 2}
For $\beta\in(k+1,m+1]$ there exists $x\in(0,\frac{m}{\beta-1})$ with a unique $\beta$-expansion.
\end{prop}
\begin{proof}
It suffices to show that there exists $x\in (0,\frac{m}{\beta-1})$ that never maps into a choice interval. We consider the point $\frac{k}{\beta-1},$ we will show that this point has a unique $\beta$-expansion. This point is contained in the $k$-th digit interval and is the fixed point under the map $T_{\beta,k}.$ To show that it has a unique $\beta$-expansion it suffices to show that it is not contained within the $(k-1)$-th or $(k+1)$-th digit intervals, this is equivalent to $$\frac{(k-1)\beta+m-(k-1)}{\beta(\beta-1)}<\frac{k}{\beta-1}<\frac{k+1}{\beta}.$$ Both of these inequalities are equivalent to $\beta>k+1.$
\end{proof}
Figure \ref{fig5} describes the construction of our point with unique $\beta$-expansion for $\beta\in(k+1,m+1]$. By Proposition \ref{Even prop 1} and Proposition \ref{Even prop 2} we can conclude Theorem \ref{Main thm} in the case where $m$ is even.
\begin{figure}[t]
\centering \unitlength=0.50mm
\begin{picture}(160,160)(0,-10)
\thinlines
\path(30,0)(0,0)(0,160)(160,160)(160,0)(130,0)
\path(40,0)(60,0)
\path(70,0)(90,0)
\path(100,0)(120,0)
\path(0,0)(160,160)
\dottedline(80,80)(80,0)
\dottedline(40,160)(40,0)
\dottedline(70,160)(70,0)
\dottedline(100,160)(100,0)
\dottedline(130,160)(130,0)
\dottedline(30,0)(40,0)
\dottedline(60,0)(70,0)
\dottedline(90,0)(100,0)
\dottedline(120,0)(130,0)

\thicklines
\path(0,0)(40,160)
\path(30,0)(70,160)
\path(60,0)(100,160)
\path(90,0)(130,160)
\path(120,0)(160,160)
\path(80,2)(80,-2)

\put(80,80){\circle*{3}}
\put(-2,-8){$0$}
\put(75,-10){$\frac{k}{\beta-1}$}
\put(155,-8){$\frac{m}{\beta-1}$}
\end{picture}
\caption{A point with unique $\beta$-expansion for $\beta\in(k+1,m+1]$.}
    \label{fig5}
\end{figure}
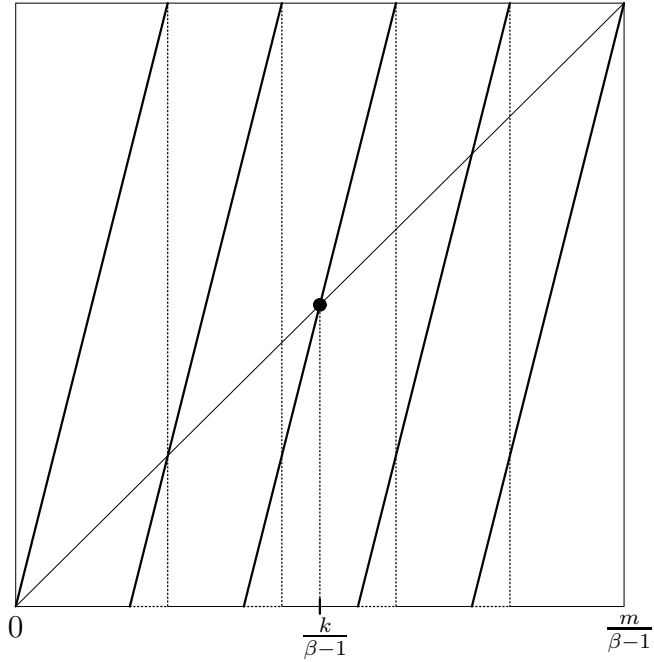

\subsection{Case where $m$ is odd}
The analysis of the case where $m$ is odd is somewhat more intricate. In what follows we assume $m=2k+1$ for some $k\in \mathbb{N}.$ Before finishing our proof of Theorem \ref{Main thm} we require the following technical results.

\begin{lem}
\label{Fixed point lemma}
For $\beta\in(1,k+2)$ the fixed point of $T_{\beta,i}$ is contained in the interior of the choice interval $[\frac{i}{\beta},\frac{(i-1)\beta+m-(i-1)}{\beta(\beta-1)}]$ for $i\in\{1,\ldots,k\},$ and in the interior of the choice interval $[\frac{i+1}{\beta},\frac{i\beta+m-i}{\beta(\beta-1)}]$ for $i\in\{k+1,\ldots,m-1\}.$
\end{lem}
\begin{proof}
Let $i\in\{1,\ldots,k\}$. To show that the fixed point $\frac{i}{\beta-1}$ is contained in the interior of the interval $[\frac{i}{\beta},\frac{(i-1)\beta+m-(i-1)}{\beta(\beta-1)}]$ it suffices to show that $$\frac{i}{\beta-1}<\frac{(i-1)\beta+m-(i-1)}{\beta(\beta-1)}.$$ This is equivalent to $\beta<m+1-i$, which for $\beta\in(1,k+2)$ is true for all $i\in\{1,\ldots,k\}$. The case where $i\in\{k+1,\ldots,m-1\}$ is proved similarly.
\end{proof} 
\begin{cor}
\label{Increasing cor}
For $\beta\in[\frac{2k+3}{2},k+2)$ the map $T_{\beta,i}$ satisfies $T_{\beta,i}(x)-\frac{i}{\beta-1}=\beta(x-\frac{i}{\beta-1})$ for all $x$ contained in the $i$-th fixed digit interval for $i\in\{1,\ldots,k\},$ and $\frac{i}{\beta-1}-T_{\beta,i}(x)=\beta(\frac{i}{\beta-1}-x)$ for all $x$ contained in the $i$-th fixed digit interval for $i\in\{k+1,\ldots,m-1\}.$
\end{cor}
\begin{proof}
Let $i\in\{1,\ldots,k\}$, by Lemma \ref{Fixed point lemma} the $i$-th fixed digit interval is to the right of the fixed point of $T_{\beta,i}$, our result follows from Lemma \ref{Map properties}. The case where $i\in\{k+1,\ldots,m-1\}$ is proved similarly.
\end{proof}
\begin{lem}
\label{jump centre}
Suppose $\beta\in[\frac{2k+3}{2},\frac{k+1+\sqrt{k^{2}+6k+5}}{2})$ and $x$ is an element of the $i$-th fixed digit interval for some $i\in\{1,\ldots,m-1\}.$ For $i\in\{1,\ldots,k\}$ $$T_{\beta,i}(x)<\frac{k\beta+m-k}{\beta(\beta-1)}$$ and for $i\in\{k+1,\ldots,m-1\}$ $$T_{\beta,i}(x)>\frac{k+1}{\beta}.$$
\end{lem}
\begin{proof}
By the monotonicity of the maps $T_{\beta,i}$ it is sufficient to show that $$T_{\beta,i}\Big(\frac{i+1}{\beta}\Big)< \frac{k\beta+m-k}{\beta(\beta-1)}$$ for $i\in\{1,\ldots,k\},$ and $$T_{\beta,i}\Big(\frac{(i-1)\beta+m-(i-1)}{\beta(\beta-1)}\Big)>\frac{k+1}{\beta},$$for $i\in\{k+1,\ldots,m-1\}.$ Each of these inequalties are equivalent to $\beta^{2}-(k+1)\beta-(k+1)<0.$ Our result follows by an application of the quadratic formula.
\end{proof}

\begin{prop}
\label{Odd prop 1}
For $\beta\in(1,\frac{k+1+\sqrt{k^{2}+6k+5}}{2})$ every $x\in(0,\frac{m}{\beta-1})$ has uncountably many $\beta$-expansions.
\end{prop}
\begin{proof}
The proof where $\beta\in(1,\frac{2k+3}{2})$ is analogous to that given in the even case. As such, in what follows we assume $\beta\in[\frac{2k+3}{2},\frac{k+1+\sqrt{k^{2}+6k+5}}{2}).$ We remark that $$\frac{k+1+\sqrt{k^{2}+6k+5}}{2}\leq \frac{m+\sqrt{m^{2}+4}}{2}$$ and $$\frac{k+1+\sqrt{k^{2}+6k+5}}{2}<k+2,$$ for all $k\in\mathbb{N}.$ We can therefore use Lemma \ref{switch lemma} and Corollary \ref{Increasing cor}. Let $x\in (0,\frac{m}{\beta-1}),$ we will show that there exists a sequence of maps that map $x$ into the interior of a choice interval, by Lemma \ref{Bijection lemma} and Proposition \ref{Uncountable proposition} our result follows. By Lemma \ref{switch lemma} there exist a finite sequence of maps that map $x$ into the interior of the switch region. Suppose the image of $x$ is not contained in the interior of a choice interval, then it must be contained in the $i$-th fixed digit interval for some $i\in\{1,\ldots,m-1\}.$ By repeatedly applying Corollary \ref{Increasing cor} and Lemma $\ref{jump centre}$ the image of $x$ must eventually be mapped into the interior of a choice interval.
\end{proof}
We refer the reader to Figure \ref{Fig2} for a diagram illustrating the case where $m=2k+1$ and $\beta\in[\frac{2k+3}{2},\frac{k+1+\sqrt{k^{2}+6k+5}}{2}).$
\begin{figure}[t]
\centering \unitlength=0.6mm
\begin{picture}(150,150)(0,-12)
\thinlines
\path(35,0)(0,0)(0,150)(150,150)(150,0)(115,0)
\path(45,0)(70,0)
\path(105,0)(80,0)
\path(51,2)(53,0)(51,-2)
\path(61,2)(63,0)(61,-2)
\path(99,2)(97,0)(99,-2)
\path(89,2)(87,0)(89,-2)
\put(-1,-6){$0$}
\put(145,-7){$\frac{m}{\beta-1}$}
\put(33,-8){$\frac{1}{\beta}$}
\put(104,-9){$\frac{(m-1)\beta+1}{\beta(\beta-1)}$}

\put(61,-9){$\frac{k+1}{\beta}$}
\put(76,-9){$\frac{k\beta+m-k}{\beta(\beta-1)}$}

\thicklines
\path(0,0)(45,150)
\path(35,0)(80,150)
\path(70,0)(115,150)
\path(105,0)(150,150)
\path(35,2)(35,-2)
\path(115,2)(115,-2)
\path(70,2)(70,-2)
\path(80,2)(80,-2)
\path(45,2)(45,-2)
\path(105,2)(105,-2)
\dottedline(45,150)(45,0)
\dottedline(80,150)(80,0)
\dottedline(115,150)(115,0)
\dottedline(35,0)(45,0)
\dottedline(70,0)(80,0)
\dottedline(105,0)(115,0)

\end{picture}
\caption{A diagram of the case where $m=2k+1$ and $\beta\in[\frac{2k+3}{2},\frac{k+1+\sqrt{k^{2}+6k+5}}{2})$}
   \label{Fig2}
\end{figure}
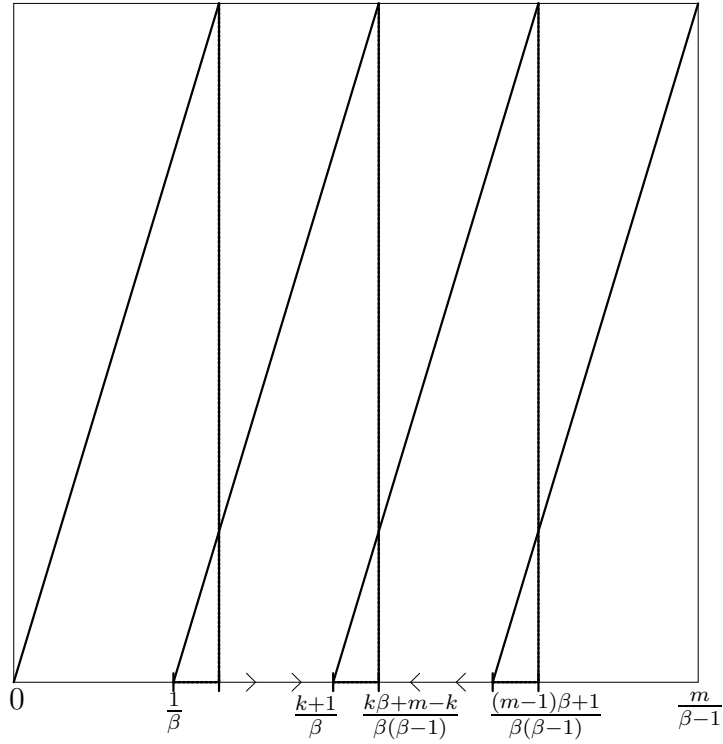
\begin{prop}
\label{Odd prop 2}
For $\beta\in(\frac{k+1+\sqrt{k^{2}+6k+5}}{2},m+1]$ there exists $x\in (0,\frac{m}{\beta-1})$ that has a unique $\beta$-expansion. 
\end{prop}
\begin{proof}
We will show that the points $$\frac{k\beta+(k+1)}{\beta^{2}-1} \textrm{ and } \frac{(k+1)\beta +k}{\beta^{2}-1}$$ have a unique $\beta$-expansion. The significance of these points is that $$T_{\beta,k}\Big(\frac{k\beta+(k+1)}{\beta^{2}-1}\Big)=\frac{(k+1)\beta +k}{\beta^{2}-1}$$ and $$T_{\beta,k+1}\Big(\frac{(k+1)\beta +k}{\beta^{2}-1}\Big)=\frac{k\beta+(k+1)}{\beta^{2}-1}.$$To show that these points have a unique $\beta$-expansion it suffices to show that 
\begin{equation}
\label{equation 1}
\frac{(k-1)\beta+m-(k-1)}{\beta(\beta-1)}<\frac{k\beta +(k+1)}{\beta^{2}-1}<\frac{k+1}{\beta},
\end{equation}
and 
\begin{equation}
\label{equation 2}
\frac{k\beta+(m-k)}{\beta(\beta-1)}<\frac{(k+1)\beta +k}{\beta^{2}-1}<\frac{k+2}{\beta}.
\end{equation}
The left hand side of (\ref{equation 1}) is equivalent to $0<\beta^{2}-k\beta-(k+2)$ which is equivalent to $$\frac{k+\sqrt{k^{2}+4k+8}}{2}<\beta,$$ however $$\frac{k+\sqrt{k^{2}+4k+8}}{2}<\frac{k+1+\sqrt{k^{2}+6k+5}}{2}$$ for all $k\in \mathbb{N},$ therefore the left hand side of (\ref{equation 1}) holds. The right hand side of (\ref{equation 1}) is equivalent to $0<\beta^{2}-(k+1)\beta-(k+1).$ So (\ref{equation 1}) holds by the quadratic formula.

The right hand side of (\ref{equation 2}) is equivalent to $0<\beta^{2}-k\beta-(k+2)$ which we know to be true by the above. Similarly the left hand side of (\ref{equation 2}) is equivalent to $0<\beta^{2}-(k+1)\beta-(k+1),$ which we also know to be true. It follows that both $\frac{k\beta+(k+1)}{\beta^{2}-1}$ and $\frac{(k+1)\beta +k}{\beta^{2}-1}$ are never mapped into a choice interval and have a unique $\beta$-expansion for $\beta\in(\frac{k+1+\sqrt{k^{2}+6k+5}}{2},m+1].$
\end{proof}
We refer the reader to Figure \ref{fig6} for a diagram describing the points we constructed with unique $\beta$-expansion for $\beta\in(\frac{k+1+\sqrt{k^{2}+6k+5}}{2},m+1].$ By Proposition \ref{Odd prop 1} and Proposition \ref{Odd prop 2} we can conclude Theorem \ref{Main thm}. 

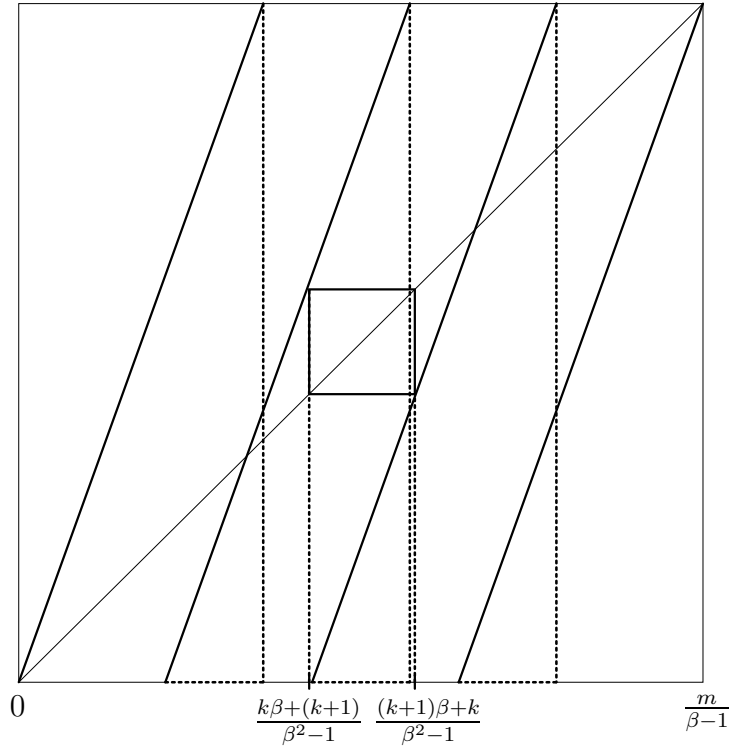
\begin{figure}[t]
\centering \unitlength=0.9mm
\begin{picture}(100,100)(0,-8)
\thinlines
\path(21.4284,0)(0,0)(0,100)(100,100)(100,0)(78.5716,0)
\path(35.7138,0)(42.8568,0)
\path(57.142,0)(64.2852,0)
\path(0,0)(100,100)
\thicklines
\path(0,0)(35.7138,100)
\path(21.4282,0)(57.142,100)
\path(42.8568,0)(78.5706,100)
\path(64.2852,0)(100,100)
\path(42.43052632,-1)(42.43052632,1)
\path(57.8947368,-1)(57.8947368,1)
\path(42.43052632,57.8947368)(57.8947368,57.8947368)(57.8947368,42.43052632)(42.43052632,42.43052632)(42.43052632,57.8947368)
\dottedline(21.4282,0)(35.7138,0)
\dottedline(42.8568,0)(57.142,0)
\dottedline(64.2852,0)(78.5706,0)

\dottedline(42.43052632,0)(42.43052632,57.8947368)
\dottedline(57.8947368,0)(57.8947368,42.43052632)
\dottedline(35.7138,100)(35.7138,0)
\dottedline(57.142,100)(57.142,0)
\dottedline(78.5706,100)(78.5706,0)

\put(-1.3,-5){$0$}
\put(97,-5){$\frac{m}{\beta-1}$}
\put(34.43052632,-7){$\frac{k\beta+(k+1)}{\beta^{2}-1}$}
\put(51.8947368,-7){$\frac{(k+1)\beta+k}{\beta^{2}-1}$}
\end{picture}
\caption{A point with unique $\beta$-expansion for $\beta\in(\frac{k+1+\sqrt{k^{2}+6k+5}}{2},m+1]$.}
    \label{fig6}
\end{figure}

\section{The set of points with unique $\beta$-expansion}
In this section we study the set of points whose $\beta$-expansion is unique for $\beta\in(\mathcal{G}(m),m+1]$. Let $$U_{\beta,m}=\Big\{x\in I_{\beta,m}|\textrm{ } \left|\Sigma_{\beta,m}(x)\right|=1\Big\}$$and $$W_{\beta,m}=\Big\{x\in \Big(\frac{m+1-\beta}{\beta-1},1\Big)|\textrm{ } \left|\Sigma_{\beta,m}(x)\right|=1\Big\}.$$ The significance of the set $W_{\beta,m}$ is that if $x\in U_{\beta,m},$ then it is a preimage of an element of $W_{\beta,m}.$ In \cite{GlenSid} the authors study the case where $m=1,$ they show that the following theorems hold.

\begin{thm}
\label{GlenSid theorem 1}
The set $U_{\beta,1}$ satisfies the following:
\begin{enumerate}
\item $\left|U_{\beta,1}\right|=\aleph_{0}$ for $\beta\in(\frac{1+\sqrt{5}}{2},\beta_{c})$
\item $\left|U_{\beta,1}\right|=2^{\aleph_{0}}$ for $\beta=\beta_{c}$
\item $U_{\beta,1}$ is a set of positive Hausdorff dimension for $\beta\in(\beta_{c},2].$
\end{enumerate}
\end{thm}
\begin{thm}
\label{GlenSid theorem 2}
The set $W_{\beta,1}$ satisfies the following:
\begin{enumerate}
\item $\left|W_{\beta,1}\right|=2$ for $\beta\in(\frac{1+\sqrt{5}}{2},\beta_{f}],$  where $\beta_{f}$ is the root of the equation $$x^{3}-2x^{2}+x-1=0,\textrm{ } \beta_{f}=1.75487\ldots$$
\item $\left|W_{\beta,1}\right|=\aleph_{0}$ for $\beta\in(\beta_{f},\beta_{c})$
\item $\left|W_{\beta,1}\right|=2^{\aleph_{0}}$ for $\beta=\beta_{c}$
\item $W_{\beta,1}$ is a set of positive Hausdorff dimension for $\beta\in(\beta_{c},2].$
\end{enumerate}
\end{thm}

Here $\beta_{c}\approx 1.78723$ is the Komornik-Loreti constant introduced in \cite{KomLor}. It is the smallest value of $\beta$ for which $1\in U_{\beta,1}$. Moreover $\beta_{c}$ is the unique solution of the equation $$\sum_{i=1}^{\infty}\frac{\lambda_{i}}{\beta^{i}}=1,$$ where $(\lambda_{i})_{i=0}^{\infty}$ is the Thue-Morse sequence (see \cite{AllShall}), i.e. $\lambda_{0}=0$ and if $\lambda_{i}$ is already defined for some $i\geq 0$ then $\lambda_{2i}=\lambda_{i}$ and $\lambda_{2i+1}=1-\lambda_{i}.$ The sequence $(\lambda_{i})_{i=0}^{\infty}$ begins $$(\lambda_{i})_{i=0}^{\infty}=0110\textrm{ }1001\textrm{ }1001\textrm{ }0110\textrm{ }1001\textrm{ }\ldots.$$ In \cite{AllCos} it was shown that $\beta_{c}$ is transcendental. For $m\geq2$ we define the sequence $(\lambda_{i}(m))_{i=1}^{\infty}\in\{0,\ldots,m\}^{\mathbb{N}}$ as follows:
\[ \lambda_{i}(m) = \left\{ \begin{array}{ll}
  k+\lambda_{i}-\lambda_{i-1}  & \mbox{if $m=2k$}\\
        k+\lambda_{i}  & \mbox{if $m=2k+1$.}\end{array} \right. \]We define $\beta_{c}(m)$ to be the unique solution of $$\sum_{i=1}^{\infty}\frac{\lambda_{i}(m)}{\beta^{i}}=1.$$ In \cite{KomLor2} the authors proved that $\beta_{c}(m)$ is transcendental and the smallest value of $\beta$ for which $1\in U_{\beta,m}.$ In section 6 we include a table of values for $\beta_{c}(m)$. We begin our study of the sets $U_{\beta,m}$ and $W_{\beta,m}$ by showing that the following proposition holds.
\begin{prop}
\label{Countable proposition}
Let $m\geq 2,$ then $\left|U_{\beta,m}\right|\geq \aleph_{0}$ for $\beta\in(\mathcal{G}(m),m+1].$ 
\end{prop}
     
Combining Proposition \ref{Countable proposition} with the results presented in \cite{KoLiDe} the following analogue of Theorem \ref{GlenSid theorem 1} is immediate.  
\begin{thm}
\label{Uniqueness theorem}
For $m\geq 2$ the set $U_{\beta,m}$ satisfies the following:
\begin{enumerate}
\item $\left|U_{\beta,m}\right|=\aleph_{0}$ for $\beta\in(\mathcal{G}(m),\beta_{c}(m))$
\item $\left|U_{\beta,m}\right|=2^{\aleph_{0}}$ for $\beta=\beta_{c}(m)$
\item $U_{\beta,m}$ is a set of positive Hausdorff dimension for $\beta\in(\beta_{c}(m),m+1].$
\end{enumerate}

\begin{proof}[Proof of Proposition \ref{Countable proposition}]
To begin with let us assume $m=2k$ for some $k\in\mathbb{N},$ in this case $\mathcal{G}(m)=k+1.$ It is a simple exercise to show that for $\beta\in(k+1,m+1]$ 
\begin{equation}
\label{Never map equation}
T^{-n}_{\beta,0}\Big(\frac{k}{\beta-1}\Big)=\frac{k}{\beta^{n}(\beta-1)}<\frac{1}{\beta}
\end{equation} for all $n\in\mathbb{N}.$ By Proposition \ref{Even prop 2} we know that $\frac{k}{\beta-1}$ has a unique $\beta$-expansion. It follows from (\ref{Never map equation}) that $T^{-n}_{\beta,0}(\frac{k}{\beta-1})$ is never mapped into a choice interval and therefore has a unique $\beta$-expansion. As $n$ was arbitrary we can conclude our result. The case where $m=2k+1$ is proved similarly, in this case we can consider preimages of $\frac{k\beta+(k+1)}{\beta^{2}-1}.$
\end{proof} 
\end{thm}
We also show that the following analogue of Theorem \ref{GlenSid theorem 2} holds.
\begin{thm}
\label{Attractor thm}
If $m=2k$ the set $W_{\beta,m}$ satisfies the following:
\begin{enumerate}
\item $\left|W_{\beta,m}\right|=1$ for $\beta\in(\mathcal{G}(m),\beta_{f}(m)],$  where $\beta_{f}(m)$ is the root of the equation $$x^2-(k+1)x-k=0,\textrm{ } \beta_{f}(m)=\frac{k+1+\sqrt{k^{2}+6k+1}}{2}$$
\item $\left|W_{\beta,m}\right|=\aleph_{0}$ for $\beta\in(\beta_{f}(m),\beta_{c}(m))$
\item $\left|W_{\beta,m}\right|=2^{\aleph_{0}}$ for $\beta=\beta_{c}(m)$
\item $W_{\beta,m}$ is a set of positive Hausdorff dimension for $\beta\in(\beta_{c}(m),m+1].$
\end{enumerate}
If $m=2k+1$ the set $W_{\beta,m}$ satisfies the following:
\begin{enumerate}
\item $\left|W_{\beta,m}\right|=2$ for $\beta\in(\mathcal{G}(m),\beta_{f}(m)],$  where $\beta_{f}(m)$ is the root of the equation $$x^3-(k+2)x^{2}+x-(k+1)=0$$
\item $\left|W_{\beta,m}\right|=\aleph_{0}$ for $\beta\in(\beta_{f}(m),\beta_{c}(m))$
\item $\left|W_{\beta,m}\right|=2^{\aleph_{0}}$ for $\beta=\beta_{c}(m)$
\item $W_{\beta,m}$ is a set of positive Hausdorff dimension for $\beta\in(\beta_{c}(m),m+1].$
\end{enumerate}
\end{thm}
\begin{remark}
$\beta_{f}(m)$ is a Pisot number for all $m\in\mathbb{N}.$
\end{remark}
Using Theorem \ref{Uniqueness theorem}, to prove Theorem \ref{Attractor thm} it suffices to show that statement $1$ holds in both the odd and even cases and $\left|W_{\beta,m}\right|\geq \aleph_{0}$ for $\beta>\beta_{f}(m)$ in both the odd and even cases. In section 6 we include a table of values for $\beta_{f}(m).$

\subsection{Proof of Theorem \ref{Attractor thm}}
The proof of Theorem \ref{Attractor thm} is more involved than Theorem \ref{Uniqueness theorem} and as we will see requires more technical results. The following is taken from \cite{KoLiDe}. Firstly let us define the lexicographic order on $\{0,\ldots,m\}^{\mathbb{N}},$ we say that $(x_{i})_{i=1}^{\infty}<(y_{i})_{i=1}^{\infty}$ with respect to the lexicographic order if there exists $n\in\mathbb{N}$ such that $x_{i}=y_{i}$ for all $i<n$ and $x_{n}<y_{n}$ or if $x_{1}<y_{1}$. For a sequence $(x_{i})_{i=1}^{\infty}\in\{0,\ldots,m\}^{\mathbb{N}}$ we define $(\bar{x}_{i})_{i=1}^{\infty}=(m-x_{i})_{i=1}^{\infty}.$ We also adopt the notation $(\epsilon_{1},\ldots,\epsilon_{j})^{\infty}$ to denote the element of $\{0,\ldots,m\}^{\mathbb{N}}$ obtained by the infinite concatenation of the finite sequence $(\epsilon_{1},\ldots,\epsilon_{j}).$ Let the sequence $(d_{i}(m))_{i=1}^{\infty}\in\{0,\ldots,m\}^{\mathbb{N}}$ be defined as follows: let $d_{1}(m)$ be the largest element of $\{0,\ldots,m\}$ such that $\frac{d_{1}(m)}{\beta}<1,$ and if $d_{i}(m)$ is defined for $i<n$ then $d_{n}(m)$ is defined to be the largest element of $\{0,\ldots,m\}$ such that $\sum_{i=1}^{n}\frac{d_{i}(m)}{\beta^i}<1.$ The sequence $(d_{i}(m))_{i=1}^{\infty}$ is called the quasi-greedy expansion of $1$ with respect to $\beta$; it is trivially a $\beta$-expansion for $1$ and the largest infinite $\beta$-expansion of $1$ with respect to the lexicographic order not ending with $(0)^{\infty}$. We let $$S_{\beta,m}=\Big\{(\epsilon_{i})_{i=1}^{\infty}\in\{0,\ldots,m\}^{\mathbb{N}}: \sum_{i=1}^{\infty}\frac{\epsilon_{i}}{\beta^{i}}\in W_{\beta,m}\Big\},$$ it follows from the definition of $W_{\beta,m}$ that $\left|W_{\beta,m}\right|=\left|S_{\beta,m}\right|$ and to prove Theorem \ref{Attractor thm} it suffices to show that equivalent statements hold for $S_{\beta,m}.$ The following lemma which is essentially due to Parry \cite{Parry} provides a useful characterisation of $S_{\beta,m}.$

\begin{lem}
\label{Attractor lemma}
\begin{align*}
S_{\beta,m}=\Big\{&(\epsilon_{i})_{i=1}^{\infty}\in\{0,\ldots,m\}^{\mathbb{N}}: (\epsilon_{i},\epsilon_{i+1},\ldots)<(d_{1,m},d_{2,m},\ldots) \textrm{ and }\\ &(\bar{d}_{1,m},\bar{d}_{2,m},\ldots)<(\epsilon_{i},\epsilon_{i+1},\ldots) \textrm{ for all } i\in\mathbb{N}\Big\}
\end{align*}
\end{lem}
\begin{remark}
\label{Beta inclusion remark}
If $\beta< \beta'$ then the quasi-greedy expansion of $1$ with respect to $\beta$ is lexicographically strictly less than the quasi-greedy expansion of $1$ with respect to $\beta'.$ As a corollary of this we have $S_{\beta,m}\subseteq S_{\beta',m}$ for $\beta<\beta'.$  
\end{remark}
\begin{prop}
For $\beta\in(\mathcal{G}(m),\beta_{f}(m)]$ $\left|S_{\beta,m}\right|=1$ when $m$ is even, $\left|S_{\beta,m}\right|=2$ when $m$ is odd and $\left|S_{\beta,m}\right|\geq \aleph_{0}$ for $\beta\in(\beta_{f}(m),m+1].$
\end{prop}By the remarks following Theorem \ref{Attractor thm} this will allow us to conclude our result.

\begin{proof}
We begin by considering the case where $m=2k$. When $\beta=\beta_{f}(m)$ we have $(d_{i}(m))_{i=1}^{\infty}=(k+1,k-1)^{\infty}$ and by Lemma \ref{Attractor lemma}
\begin{align*}
S_{\beta_{f}(m),m}=\Big\{&(\epsilon_{i})_{i=1}^{\infty}\in\{0,\ldots,m\}^{\mathbb{N}}: (\epsilon_{i},\epsilon_{i+1},\ldots)<(k+1,k-1)^{\infty}\textrm{ and }\\
&(k-1,k+1)^{\infty}<(\epsilon_{i},\epsilon_{i+1},\ldots) \textrm{ for all } i\in\mathbb{N}\Big\}.
\end{align*}
By our previous analysis we know that for $\beta\in(\mathcal{G}(m),m+1]$ the point $\frac{k}{\beta-1}$ has a unique $\beta$-expansion, the $\beta$-expansion of this point is the sequence $(k)^{\infty}.$ By Remark \ref{Beta inclusion remark}, to prove $\left|S_{\beta,m}\right|=1$ for $\beta\in(\mathcal{G}(m),\beta_{f}(m)]$ it suffices to show that $S_{\beta_{f}(m),m}=\{(k)^{\infty}\}.$ Let $(\epsilon_{i})_{i=1}^{\infty}\in S_{\beta_{f}(m),m},$ clearly $\epsilon_{i}$ must equal $k-1,k$ or $k+1.$ If $\epsilon_{i}=k+1$ then by Lemma \ref{Attractor lemma} $\epsilon_{i+1}=k-1,$ similarly if $\epsilon_{i}=k-1$ then $\epsilon_{i+1}=k+1.$ Therefore if $\epsilon_{i}\neq k$ for some $i,$ then $(\epsilon_{i},\epsilon_{i+1},\ldots)$ must equal $(k-1,k+1)^{\infty}$ or $(k+1,k-1)^{\infty}.$ By Lemma \ref{Attractor lemma} this cannot happen and we can conclude that $S_{\beta_{f}(m),m}=\{(k)^{\infty}\}.$ For $\beta\in(\beta_{f,m},m+1],$ we can construct a countable subset of $S_{\beta,m}$; for example all sequences of the form $(k)^{j}(k+1,k-1)^{\infty}$ where $j\in \mathbb{N}.$ 

We now consider the case where $m=2k+1$, when $\beta=\beta_{f}(m)$ we have $(d_{i}(m))_{i=1}^{\infty}=(k+1,k+1,k,k)^{\infty}$ and 
\begin{align*}
S_{\beta_{f}(m),m}=\Big\{&(\epsilon_{i})_{i=1}^{\infty}\in\{0,\ldots,m\}^{\mathbb{N}}: (\epsilon_{i},\epsilon_{i+1},\ldots)<(k+1,k+1,k,k)^{\infty}\textrm{ and }\\
&(k,k,k+1,k+1)^{\infty}<(\epsilon_{i},\epsilon_{i+1},\ldots) \textrm{ for all } i\in\mathbb{N}\Big\}.
\end{align*} By our earlier analysis we know that $\{(k,k+1)^{\infty},(k+1,k)^{\infty}\}\subset S_{\beta,m}$ for $\beta\in(\mathcal{G}(m),m+1].$ By Remark \ref{Beta inclusion remark} to prove $\left|S_{\beta,m}\right|=2$ for $\beta\in(\mathcal{G}(m),\beta_{f}(m)]$ it suffices to show that $S_{\beta_{f}(m),m}=\{(k,k+1)^{\infty},(k+1,k)^{\infty}\}.$ By an analogous argument to that given in \cite{GlenSid} we can show that if $(\epsilon_{i})_{i=1}^{\infty}\in S_{\beta_{f}(m),m}$ then $\epsilon_{i}=k$ implies $\epsilon_{i+1}=k+1,$ and $\epsilon_{i}=k+1$ implies $\epsilon_{i+1}=k.$ Clearly any element of $S_{\beta(f)(m),m}$ must begin with $k$ or $k+1$ and we may therefore conclude that $S_{\beta_{f}(m),m}=\{(k,k+1)^{\infty},(k+1,k)^{\infty}\}.$ To see that $\left|W_{\beta,m}\right|\geq \aleph_{0}$ for $\beta>\beta_{f}(m)$ we observe that $(k+1,k)^{j}(k+1,k+1,k,k)^{\infty}\in S_{\beta,m}$ for all $j\in\mathbb{N},$ for $\beta>\beta_{f}(m).$ 
\end{proof}

\subsection{The growth rate of $\mathcal{G}(m),$ $\beta_{f}(m)$ and $\beta_{c}(m)$}
In this section we study the growth rate of the sequences $(\mathcal{G}(m))_{m=1}^{\infty},$ $(\beta_{f}(m))_{m=1}^{\infty}$ and $(\beta_{c}(m))_{m=1}^{\infty}.$ The following theorem summarises the growth rate of each of these sequences.

\begin{thm}
\label{Growth rate thm}
\begin{enumerate}
\item $\mathcal{G}(2k)=k+1$ for all $k\in\mathbb{N}.$
\item $\beta_{f}(2k)-(k+2)=O(\frac{1}{k}).$
\item $\beta_{c}(2k)-(k+2)\to 0$ as $k\to \infty$.
\item $\mathcal{G}(2k+1)-(k+2)=O(\frac{1}{k}).$ 
\item $\beta_{f}(2k+1)-(k+2)\to 0$ as $k\to\infty.$
\item $\beta_{c}(2k+1)-(k+2)\to 0$ as $k\to \infty$.
\end{enumerate}
\end{thm}
The proof of this theorem is somewhat trivial but we include it for completion. To prove this result we firstly require the following lemma.

\begin{lem}
\label{Growth lemma}
The sequence $\beta_{c}(m)$ is asymptotic to $\frac{m}{2},$ i.e., $\lim_{m\to\infty} \frac{\beta_{c}(m)}{m/2}=1.$
\end{lem}
\begin{proof}
Suppose $m=2k$. It is a direct consequence of the definition of $\lambda_{i}(m)$ and $\beta_{c}(m)$ that the following inequalties hold 
$$\sum_{i=0}^{\infty}\frac{k-1}{\beta_{c}(m)^{i}}\leq \beta_{c}(m) \leq \sum_{i=0}^{\infty}\frac{k+1}{\beta_{c}(m)^{i}},$$ which is equivalent to 
$$\frac{k-1}{1-\frac{1}{\beta_{c}(m)}}\leq \beta_{c}(m)\leq \frac{k+1}{1-\frac{1}{\beta_{c}(m)}}.$$ Dividing through by $m/2$ and using the fact that $\beta_{c}(m)\to\infty$ we can conclude our result. The case where $m=2k+1$ is proved similarly.
\end{proof}
We are now in a position to prove Theorem \ref{Growth rate thm}.
\begin{proof}[Proof of Theorem \ref{Growth rate thm}]
Statements $1,2$ and $4$ are an immediate consequence of Theorem \ref{Main thm} and Theorem \ref{Attractor thm}. It remains to show statements $3$ and $6$ hold; statement $4$ will follow from the fact that $\mathcal{G}(2k+1)<\beta_{f}(2k+1)<\beta_{c}(2k+1).$ It is immediate from the definition of $\lambda_{i}(m)$ that if $m=2k$ then $$\beta_{c,m}=k+1+\frac{k}{\beta_{c}(m)}+\sum_{i=2}^{\infty}\frac{\lambda_{i+1}(m)}{\beta^{i}}.$$ Our result now follows from Lemma \ref{Growth lemma} and the fact that $\sum_{i=2}^{\infty}\frac{\lambda_{i+1}(m)}{\beta_{c}(m)^{i}}\to 0$ as $m\to\infty$. The case where $m=2k+1$ is proved similarly.

\end{proof}

\section{The growth rate and dimension theory of $\Sigma_{\beta,m}(x)$}
To describe the growth rate of $\beta$-expansions we consider the following. Let
\begin{align*}
\mathcal{E}_{\beta,m,n}(x)=\Big\{&(\epsilon_{1},\ldots,\epsilon_{n})\in\{0,\ldots,m\}^{n}|\exists (\epsilon_{n+1}, \epsilon_{n+2}, \ldots)\in \{0,\ldots, m\}^{\mathbb{N}}\\
&:\sum_{i=1}^{\infty}\frac{\epsilon_{i}}{\beta^{i}}=x\Big\},
\end{align*}
we define an element of $\mathcal{E}_{\beta,m,n}(x)$ to be a \textit{$n$-prefix} for $x$. Moreover, we let $$\mathcal{N}_{\beta,m,n}(x)=\left|\mathcal{E}_{\beta,m,n}(x)\right|$$and define the \textit{growth rate of $\mathcal{N}_{\beta,m,n}(x)$} to be $$\lim_{n\to\infty} \frac{\log_{m+1}\mathcal{N}_{\beta,m,n}(x)}{n},$$ when this limit exists. When this limit does not exist we can consider the \textit{lower and upper growth rates of $\mathcal{N}_{\beta,m,n}(x)$}, these are defined to be $$\liminf_{n\to\infty} \frac{\log_{m+1}\mathcal{N}_{\beta,m,n}(x)}{n}\textrm{ and }\limsup_{n\to\infty} \frac{\log_{m+1}\mathcal{N}_{\beta,m,n}(x)}{n}$$ respectively. 

In this paper we also consider $\Sigma_{\beta,m}(x)$ from a dimension theory perspective. We endow $\{0,\ldots,m\}^{\mathbb{N}}$ with the metric $d(\cdot,\cdot)$ defined as follows:
\[ d(x,y) = \left\{ \begin{array}{ll}
         (m+1)^{-n(x,y)} & \mbox{if $x\neq y,$ where $n(x,y)=\inf \{i:x_{i}\neq y_{i}\} $}\\
        0 & \mbox{if $x=y$.}\end{array} \right. \]We will consider the Hausdorff dimension of $\Sigma_{\beta,m}(x)$ with respect to this metric. It is a simple exercise to show that following inequalities hold:
\begin{equation}
\label{dimension inequality}
\dim_{H}(\Sigma_{\beta,m}(x))\leq \liminf_{n\to\infty} \frac{\log_{m+1}\mathcal{N}_{\beta,m,n}(x)}{n}\leq\limsup_{n\to\infty} \frac{\log_{m+1}\mathcal{N}_{\beta,m,n}(x)}{n}.
\end{equation} 
The case where $m=1$ is studied in \cite{Baker}, \cite{FengSid} and \cite{Kempton}. In \cite{Baker} and \cite{FengSid} the authors show that for $\beta\in(1,\frac{1+\sqrt{5}}{2})$ and $x\in(0,\frac{1}{\beta-1})$ we can bound the lower growth rate and Hausdorff dimension of $\Sigma_{\beta,1}(x)$ below by some strictly positive function depending only on $\beta$, in \cite{Kempton} the growth rate is studied from a measure theoretic perspective. Our main result is the following.

\begin{thm}
\label{Dimension theorem}
For $\beta\in(1,\mathcal{G}(m))$ and $x\in(0,\frac{m}{\beta-1})$ the Hausdorff dimension of $\Sigma_{\beta,m}(x)$ can be bounded below by some strictly positive constant depending only on $\beta.$
\end{thm} By $(\ref{dimension inequality})$ a similar statement holds for both the lower and upper growth rates of $\mathcal{N}_{\beta,m,n}(x)$. Replcating the proof of Lemma \ref{Bijection lemma} it is a simple exercise to show that the following result holds.
\begin{prop}
\label{Change of perspective}
$\mathcal{N}_{\beta,m,n}(x)=\left|\Omega_{\beta,m,n}(x)\right|$
\end{prop}
By Proposition \ref{Change of perspective} we can identify elements of $\Omega_{\beta,m,n}(x)$ with elements of $\mathcal{E}_{\beta,m,n}(x),$ as such we also define an element of $\Omega_{\beta,m,n}(x)$ to be a \textit{$n$-prefix} for $x$. To prove Theorem \ref{Dimension theorem} we will use a method analogous to that given if $\cite{Baker}$. We construct an interval $\mathcal{I}_{\beta}\subset I_{\beta,m}$ such that, for each $x\in\mathcal{I}_{\beta}$ we can generate multiple prefixes for $x$ of a fixed length depending on $\beta$ that map $x$ back into $\mathcal{I}_{\beta}$. As we will see Theorem \ref{Dimension theorem} will then follow by a counting argument. As was the case in our previous analysis we reduce the proof of Theorem \ref{Dimension theorem} to two cases.

\subsection{Case where $m$ is even}
In what follows we assume $m=2k$ for some $k\in\mathbb{N}.$ To prove Theorem \ref{Dimension theorem} we require the following technical lemma.

\begin{lem}
\label{smaller switch}
For each $\beta\in(1,k+1)$ there exists $\epsilon_{0}(\beta)>0$ such that, if $x\in[\frac{1}{\beta},\frac{1}{\beta}+\epsilon_{0}(\beta))$ then $T_{\beta,0}(x)\in[\frac{1}{\beta}+\epsilon_{0}(\beta),\frac{(m-1)\beta+1}{\beta(\beta-1)}-\epsilon_{0}(\beta)],$ and similarly if $x\in(\frac{(m-1)\beta+1}{\beta(\beta-1)}-\epsilon_{0}(\beta),\frac{(m-1)\beta+1}{\beta(\beta-1)}]$ then $T_{\beta,m}(x)\in[\frac{1}{\beta}+\epsilon_{0}(\beta),\frac{(m-1)\beta+1}{\beta(\beta-1)}-\epsilon_{0}(\beta)].$
\end{lem}
\begin{proof}
This follows from Lemma \ref{switch lemma} and a continuity argument.
\end{proof}
For each $i\in\{1,\ldots,m-1\}$ we let $\epsilon_{i}(\beta)=\frac{1}{2}(\frac{(i-1)\beta+m-(i-1)}{\beta(\beta-1)}-\frac{i+1}{\beta}).$ If $\beta\in(1,k+1)$ then $\epsilon_{i}(\beta)>0$ for each $i\in\{1,\ldots,m-1\}$. We define the interval $\mathcal{I}_{\beta}=[L(\beta),R(\beta)]$ where $L(\beta)$ and $R(\beta)$ are defined as follows: $$L(\beta)=\min\Bigg\{T_{\beta,1}\Big(\frac{1}{\beta}+\epsilon_{0}(\beta)\Big), \min_{i\in \{1,\ldots,m-1\}}T_{\beta,i+1}\Big(\frac{i+1}{\beta}+\epsilon_{i}(\beta)\Big)\Bigg\}$$and 
\begin{align*}
R(\beta)=\max\Bigg\{& T_{\beta,m-1}\Big(\frac{(m-1)\beta+1}{\beta(\beta-1)}-\epsilon_{0}(\beta)\Big), \\
& \max_{i\in \{1,\ldots,m-1\}}T_{\beta,i-1}\Big(\frac{i+1}{\beta}+\epsilon_{i}(\beta)\Big)\Bigg\}.
\end{align*}
We refer to Figure \ref{fig3} for a diagram illustrating the interval $\mathcal{I}_{\beta}$ in the case where $m=2$ and $\beta\in(1,2).$

\begin{figure}[t]
\centering \unitlength=0.65mm
\begin{picture}(150,150)(0,-10)
\thinlines
\path(35,0)(0,0)(0,150)(150,150)(150,0)(115,0)
\put(-1,-8){$L(\beta)$}
\put(139,-8){$R(\beta)$}
\put(30,-7){$\frac{1}{\beta}+\epsilon_{0}(\beta)$}
\put(65.5,-7){$\frac{2}{\beta}+\epsilon_{1}(\beta)$}
\put(93.5,-7){$\frac{\beta+1}{\beta(\beta-1)}-\epsilon_{0}(\beta)$}
\thicklines
\path(0,0)(80,150)
\path(35,0)(115,150)
\path(70,0)(150,150)
\path(70,0)(80,0)
\path(70,2)(70,-2)
\path(80,2)(80,-2)
\path(38,2)(38,-2)
\path(75,2)(75,-2)
\path(112,2)(112,-2)
\path(4,2)(4,-2)
\path(146,2)(146,-2)
\dottedline(80,150)(80,0)
\dottedline(115,150)(115,0)
\dottedline(35,0)(70,0)
\dottedline(80,0)(115,0)

\end{picture}
\caption{The interval $\mathcal{I}_{\beta}$ in the case where $m=2$ and $\beta\in(1,2)$.}
    \label{fig3}
\end{figure}
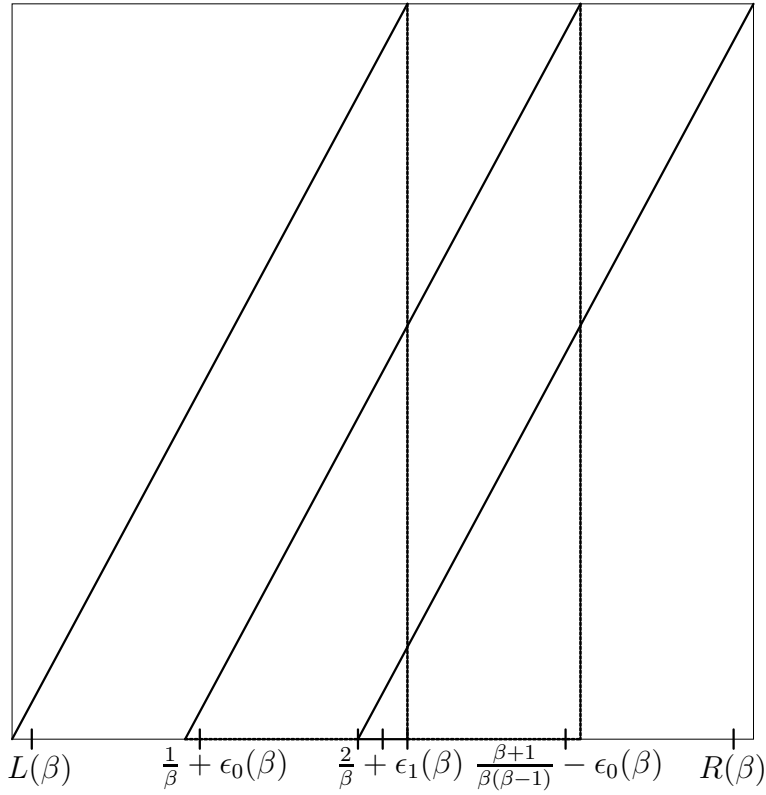

\begin{prop}
\label{Generate prefix's prop}
Let $\beta\in(1,k+1).$ There exists $n(\beta)\in\mathbb{N}$ such that, for each $x\in \mathcal{I}_{\beta}$ there exists two elements $a,b\in \Omega_{\beta,m,n(\beta)}(x)$ such that $a(x)\in \mathcal{I}_{\beta}$ and $b(x)\in \mathcal{I}_{\beta}.$
\end{prop}
\begin{proof}Let $x\in \mathcal{I}_{\beta}$. Without loss of generality we may assume that $\epsilon_{0}(\beta)$ is sufficiently small such that $\mathcal{I}_{\beta}$ contains the switch region. By Lemma \ref{switch lemma} there exists a sequence of maps $a$ that map $x$ into the interior of our switch region. By Lemma \ref{smaller switch} we may assume that $a(x)\in[\frac{1}{\beta}+\epsilon_{0}(\beta),\frac{(m-1)\beta+1}{\beta(\beta-1)}-\epsilon_{0}(\beta)].$

The distance between the endpoints of $\mathcal{I}_{\beta}$ and the endpoints of $I_{\beta,m}$ (the fixed points of the maps $T_{\beta,0}$ and $T_{\beta,m}$,) can be bounded below by some positive constant, by Lemma \ref{Map properties} $T_{\beta,0}$ and $T_{\beta,m}$ both scale the distance between their fixed points and a general point by a factor $\beta,$ therefore we can bound the length of our sequence $a$ above by some constant $n_{s}(\beta)\in\mathbb{N}$ that does not depend on $x.$ We will show that we can take $n(\beta)=n_{s}(\beta)+1$. 

We remark that

\begin{align*}
\Big[\frac{1}{\beta}+\epsilon_{0}(\beta),\frac{(m-1)\beta+1}{\beta(\beta-1)}-\epsilon_{0}(\beta)\Big]=&\Big[\frac{1}{\beta}+\epsilon_{0}(\beta),\frac{2}{\beta}\Big]\\
\bigcup&\Big[\frac{(m-2)\beta+2}{\beta(\beta-1)},\frac{(m-1)\beta+1}{\beta(\beta-1)}-\epsilon_{0}(\beta)\Big]\\
\bigcup_{i=1}^{m-2}&\Big[\frac{(i-1)\beta+m-(i-1)}{\beta(\beta-1)},\frac{i+2}{\beta}\Big]\\
\bigcup_{i=1}^{m-1}&\Big[\frac{i+1}{\beta},\frac{(i-1)\beta+m-(i-1)}{\beta(\beta-1)}\Big].
\end{align*}We now proceed via a case analysis.
\begin{itemize}
\item If $a(x)\in [\frac{1}{\beta}+\epsilon_{0}(\beta),\frac{2}{\beta}]$ then $T_{\beta,0}(a(x))\in \mathcal{I}_{\beta}$ and $T_{\beta,1}(a(x))\in\mathcal{I}_{\beta}.$ 
\item If $a(x)\in[\frac{(m-2)\beta+2}{\beta(\beta-1)},\frac{(m-1)\beta+1}{\beta(\beta-1)}-\epsilon_{0}(\beta)]$ then $T_{\beta,m-1}(a(x))\in \mathcal{I}_{\beta}$ and $T_{\beta,m}(a(x))\in\mathcal{I}_{\beta}.$ 
\item If $a(x)\in [\frac{(i-1)\beta+m-(i-1)}{\beta(\beta-1)},\frac{i+2}{\beta}]$ for some $i\in\{1,\ldots,m-2\}$ then $T_{\beta,i}(a(x))\in \mathcal{I}_{\beta}$ and $T_{\beta,i+1}(a(x))\in\mathcal{I}_{\beta}.$ 
\item We reduce the the case where $a(x)\in [\frac{i+1}{\beta},\frac{(i-1)\beta+m-(i-1)}{\beta(\beta-1)}]$ for some $i\in\{1,\ldots,m-1\}$ to two subcases. If $a(x)\in[\frac{i+1}{\beta},\frac{i+1}{\beta}+\epsilon_{i}(\beta)]$ then by the monotonicity of our maps, both $T_{\beta,i-1}(a(x))\in\mathcal{I}_{\beta}$ and $T_{\beta,i}(a(x))\in\mathcal{I}_{\beta}.$ Similarly, in the case where $a(x)\in [\frac{i+1}{\beta}+\epsilon_{i}(\beta),\frac{(i-1)\beta+m-(i-1)}{\beta(\beta-1)}]$ both $T_{\beta,i}(a(x))\in\mathcal{I}_{\beta}$ and $T_{\beta,i+1}(a(x))\in\mathcal{I}_{\beta}.$
\end{itemize}
We've shown that for any $x\in \mathcal{I}_{\beta}$ there exists $n(x)\leq n_{s}(\beta)+1$ such that two distinct elements of $\Omega_{\beta,m,n(x)}(x)$ map $x$ into $\mathcal{I}_{\beta}$. If $n(x)<n_{s}(\beta)+1$ then we can concatenate our two elements of $\Omega_{\beta,m,n(x)}(x)$ by an arbitrary choice of maps of length $n_{s}(\beta)+1-n(x)$ that map the image of $x$ into $\mathcal{I}_{\beta}.$ This ensures that we can take our sequences of maps to be of length $n_{s}(\beta)+1.$  
\end{proof}

For $\beta\in(1,k+1)$ and $x\in(0,\frac{m}{\beta-1})$ we may assume that there exists a sequence of maps $a$ that maps $x$ into $\mathcal{I}_{\beta}.$ We denote the minimum number of maps required to do this by $j(x).$ Replicating arguments given in \cite{Baker} we can use Proposition \ref{Generate prefix's prop} to construct an algorithm by which we can generate two prefixes of length $n(\beta)$ for $a^{(j(x))}.$ Repeatedly applying this algorithm to succesive images of $a^{(j(x))}$ we can generate a closed subset of $\Sigma_{\beta,m}(x)$. We denote this set by $\sigma_{\beta,m}(x)$ and the set of $n$-prefixes for $x$ generated by this algorithm by $\omega_{\beta,m,n}(x).$
Replicating the proofs given in \cite{Baker} we can show that the following lemmas hold.

\begin{lem}
\label{Minimal prefix's}
Let $x \in (0,\frac{m}{\beta-1}).$ Assume $n\geq j(x)$ then $$\left|\omega_{\beta,m,n}(x)\right|\geq 2^{\frac{n-j(x)}{n(\beta)}-1}.$$
\end{lem}

\begin{lem}
\label{Maximal prefix's}
Let $x \in (0,\frac{m}{\beta-1}).$ Assume $l\geq j(x)$ and $b\in \omega_{\beta,m,l}(x),$ then for $n\geq l$ $$\left|\{a=(a_{i})_{i=1}^{n}\in \omega_{\beta,m,n}(x): a_{i}=b_{i}\textrm{ for } 1\leq i \leq l\}\right|\leq 2^{\frac{n-l}{n(\beta)}+2}.$$
\end{lem}
With these lemmas we are now in a position to prove Theorem \ref{Dimension theorem} in the case where $m$ is even. The argument used is analogous to the one given in \cite{Baker}, which is based upon Example $2.7$ of \cite{Falconer}.
\begin{proof}[Proof of Theorem \ref{Dimension theorem} when $m=2k$] By the monotonicity of Hausdorff dimension with respect to inclusion it suffices to show that $\dim_{H}(\sigma_{\beta,m}(x))$ can be bounded below by a strictly positive constant depending only on $\beta$. It is a simple exercise to show that $\sigma_{\beta,m}(x)$ is a compact set; by this result we may restrict to finite covers of $\sigma_{\beta,m}(x).$ Let $\{U_{n}\}_{n=1}^{N}$ be a finite cover of $\sigma_{\beta,m}(x).$ Without loss of generality we may assume that all elements of our cover satisfy $\textrm{Diam}(U_{n})< (m+1)^{-j(x)}$. For each $U_{n}$ there exists $l(n)\in\mathbb{N}$ such that $$(m+1)^{-(l(n)+1)}\leq \textrm{Diam}(U_{n})< (m+1)^{-l(n)}.$$ It follows that there exists $z^{(n)}\in \{0,\ldots,m\}^{l(n)}$ such that, $y_{i}=z^{(n)}_{i}$ for $1\leq i \leq l(n),$ for all $y\in U_{n}.$ We may assume that $z^{(n)}\in \omega_{\beta,m,l(n)}(x),$ if we supposed otherwise then $\sigma_{\beta,m}(x)\cap U_{n}=\emptyset$ and we can remove $U_{n}$ from our cover. We denote by $C_{n}$ the set of sequences in $\{0,\ldots,m\}^{\mathbb{N}}$ whose first $l(n)$ entries agree with $z^{(n)},$ i.e. $$C_{n}=\Big\{(\epsilon_{i})_{i=1}^{\infty}\in \{0,\ldots,m\}^{\mathbb{N}}: \epsilon_{i}= z^{(n)}_{i}\textrm{ for } 1\leq i\leq l(n)\Big\}.$$ Clearly $U_{n}\subset C_{n}$ and therefore the set $\{C_{n}\}_{n=1}^{N}$ is a cover of $\sigma_{\beta,m}(x).$

Since there are only finitely many elements in our cover there exists $J$ such that $(m+1)^{-J}\leq \textrm{Diam}(U_{n})$ for all $n$. We consider the set $\omega_{\beta,m,J}(x).$ Since $\{C_{n}\}_{n=1}^{N}$ is a cover of $\sigma_{\beta,m}(x)$ each $a\in\omega_{\beta,m,J}(x)$ satisfies $a_{i}=z^{(n)}_{i}$ for $1\leq i \leq l(n),$ for some $n$. Therefore $$\left| \omega_{\beta,m,J}(x)\right|\leq \sum_{n=1}^{N} \left|\{a\in \omega_{\beta,m,J}(x): a_{i}=z^{(n)}_{i} \textrm{ for } 1\leq i\leq l(n)\}\right|.$$By counting elements of $\omega_{\beta,m,J}(x)$ and Lemmas \ref{Minimal prefix's} and \ref{Maximal prefix's} we observe the following;
\begin{align*}
2^{\frac{J-j(x)}{n(\beta)}-1}&\leq \left| \omega_{\beta,m,J}(x)\right|\\
&\leq \sum_{n=1}^{N} \left|\{a\in \omega_{\beta,m,J}(x): a_{i}=z^{(n)}_{i} \textrm{ for } 1\leq i\leq l(n)\}\right|\\
&\leq \sum_{n=1}^{N} 2^{\frac{J-l(n)}{n(\beta)}+2}\\
&= 2^{\frac{J+1}{n(\beta)}+2}\sum_{n=1}^{N} 2^{\frac{-(l(n)+1)}{n(\beta)}}\\
&\leq 2^{\frac{J+1}{n(\beta)}+2}\sum_{n=1}^{N} \textrm{Diam}(U_{n})^{\frac{\log_{m+1}2}{n(\beta)}}.
\end{align*} Dividing through by $2^{\frac{J+1}{n(\beta)}+2}$ yields $$\sum_{n=1}^{N} \textrm{Diam}(U_{n})^{\frac{\log_{m+1}2}{n(\beta)}} \geq 2^{\frac{-j(x)-3n(\beta)-1}{n(\beta)}},$$ the right hand side is a positive constant greater than zero that does not depend on our choice of cover. It follows that $\dim_{H}(\sigma_{\beta,m}(x))\geq\frac{\log_{m+1}2}{n(\beta)},$ our result follows.
\end{proof}

\subsection{Case where $m$ is odd}
In what follows we assume $m=2k+1$ for some $k\in\mathbb{N}.$ For $\beta\in(1,\frac{2k+3}{2})$ the proof of Theorem \ref{Dimension theorem} is analogous to the even case for $\beta\in(1,k+1).$ As such, in what follows we assume $\beta\in[\frac{2k+3}{2},\frac{k+1+\sqrt{k^{2}+6k+5}}{2}).$ The significance of $\beta\in[\frac{2k+3}{2},\frac{k+1+\sqrt{k^{2}+6k+5}}{2})$ is that for $i\in\{1,\ldots,m-1\}$ the $i$-th fixed digit interval is well defined. 

Before defining the interval $\mathcal{I}_{\beta}$ we require the following. We let 
\begin{equation*}
\epsilon_{i}(\beta) = \left\{ \begin{array}{rl}
\frac{1}{2}\Big(\frac{(i-1)\beta+m-(i-1)}{\beta(\beta-1)}-\frac{i}{\beta-1}\Big) &\mbox{ if $i\in\{1,\ldots,k\}$} \\
\frac{1}{2}\Big(\frac{i}{\beta-1}-\frac{i+1}{\beta}\Big)   &\mbox{ if $i\in\{k+1,\ldots,m-1\}$}
       \end{array} \right.
\end{equation*} By Lemma \ref{Fixed point lemma}, $\epsilon_{i}(\beta)>0$ for all $i\in\{1,\ldots,m-1\}$ for $\beta\in(1,k+2).$ 
Before proving an analogue of Proposition \ref{Generate prefix's prop} we require the following technical lemmas. It is a simple exercise to show that the following analogue of Lemma \ref{smaller switch} holds.
\begin{lem}
\label{smaller switch 2}
For each $\beta\in[\frac{2k+3}{2},\frac{k+1+\sqrt{k^{2}+6k+5}}{2})$ there exists $\epsilon_{0}(\beta)>0$ such that, if $x\in[\frac{1}{\beta},\frac{1}{\beta}+\epsilon_{0}(\beta))$ then $T_{\beta,0}(x)\in[\frac{1}{\beta}+\epsilon_{0}(\beta),\frac{(m-1)\beta+1}{\beta(\beta-1)}-\epsilon_{0}(\beta)],$ and similarly if $x\in(\frac{(m-1)\beta+1}{\beta(\beta-1)}-\epsilon_{0}(\beta),\frac{(m-1)\beta+1}{\beta(\beta-1)}]$ then $T_{\beta,m}(x)\in[\frac{1}{\beta}+\epsilon_{0}(\beta),\frac{(m-1)\beta+1}{\beta(\beta-1)}-\epsilon_{0}(\beta)].$
\end{lem}
\begin{lem}
\label{centre switch}
Let $\beta\in[\frac{2k+3}{2},\frac{k+1+\sqrt{k^{2}+6k+5}}{2}).$ For each $i\in\{1,\ldots,k-1\}$ there exists $\epsilon_{i}^{*}(\beta)>0$ such that, if $x\in [\frac{(i-1)\beta+m-(i-1)}{\beta(\beta-1)}-\epsilon_{i}(\beta),\frac{i+1}{\beta}+\epsilon_{i}^{*}(\beta)]$ then $T_{\beta,i}(x)<\frac{k+2}{\beta}+\epsilon_{k+1}.$ Similarly for $i\in\{k+2,\ldots,m-1\}$ there exists $\epsilon_{i}^{*}(\beta)>0$ such that, if $x\in [\frac{(i-1)\beta+m-(i-1)}{\beta(\beta-1)}-\epsilon_{i}^{*}(\beta),\frac{i+1}{\beta}+\epsilon_{i}(\beta)]$ then $T_{\beta,i}(x)> \frac{(k-1)\beta+m-(k-1)}{\beta(\beta-1)}-\epsilon_{k}.$
\end{lem}
\begin{proof}
By the analysis given in the proof of Lemma \ref{jump centre} for $i\in\{1,\ldots,k-1\}$ $T_{\beta,i}(\frac{i+1}{\beta})<\frac{k\beta+m-k}{\beta(\beta-1)}$ for $\beta\in(1,\frac{k+1+\sqrt{k^{2}+6k+5}}{2})$. However, for $\beta\in[\frac{2k+3}{2},\frac{k+1+\sqrt{k^{2}+6k+5}}{2})$ $\frac{k\beta+m-k}{\beta(\beta-1)}\leq \frac{k+2}{\beta}.$ The existence of $\epsilon_{i}^{*}(\beta)$ then follows by a continuity argument and the monotonicity of the maps $T_{\beta,i}$. The case where $i\in\{k+2,\ldots,m-1\}$ is proved similarly.
\end{proof}

We are now in a position to define the interval $\mathcal{I}_{\beta}$. Let $\mathcal{I}_{\beta}=[L(\beta),R(\beta)]$ where
\begin{align*} 
L(\beta)=\min & \Bigg\{T_{\beta,1}\Big(\frac{1}{\beta}+\epsilon_{0}(\beta)\Big),T_{\beta,k+1}\Big(\frac{k\beta+k+1}{\beta^{2}-1}\Big),\\
&\min_{i\in\{2,\ldots,k\}}\Big\{T_{\beta,i}\Big(\frac{i}{\beta}+\epsilon_{i-1}^{*}(\beta)\Big)\Big\},\min_{i\in\{k+2,\ldots,m\}}\Big\{T_{\beta,i}\Big(\frac{i}{\beta}+\epsilon_{i-1}(\beta)\Big)\Big\}\Bigg\}
\end{align*}
and
\begin{align*} 
R(\beta)=\max\Bigg\{&T_{\beta,k}\Big(\frac{(k+1)\beta+k}{\beta^{2}-1}\Big),T_{\beta,m-1}\Big(\frac{(m-1)\beta+1}{\beta(\beta-1)}-\epsilon_{0}(\beta)\Big),\\
&\max_{i\in\{1,\ldots,k\}}\Big\{T_{\beta,i-1}\Big(\frac{(i-1)\beta+m-(i-1)}{\beta(\beta-1)}-\epsilon_{i}(\beta)\Big)\\
&\max_{i\in\{k+2,\ldots,m-1\}}\Big\{T_{\beta,i-1}\Big(\frac{(i-1)\beta+m-(i-1)}{\beta(\beta-1)}-\epsilon_{i}^{*}(\beta)\Big)\Big\}\Bigg\}.
\end{align*}
For ease of exposition in Figure \ref{fig4} we give a diagram illustrating the interval $\mathcal{I}_{\beta},$ in the case where $m=3$ and $\beta\in[\frac{5}{2},1+\sqrt{3}).$
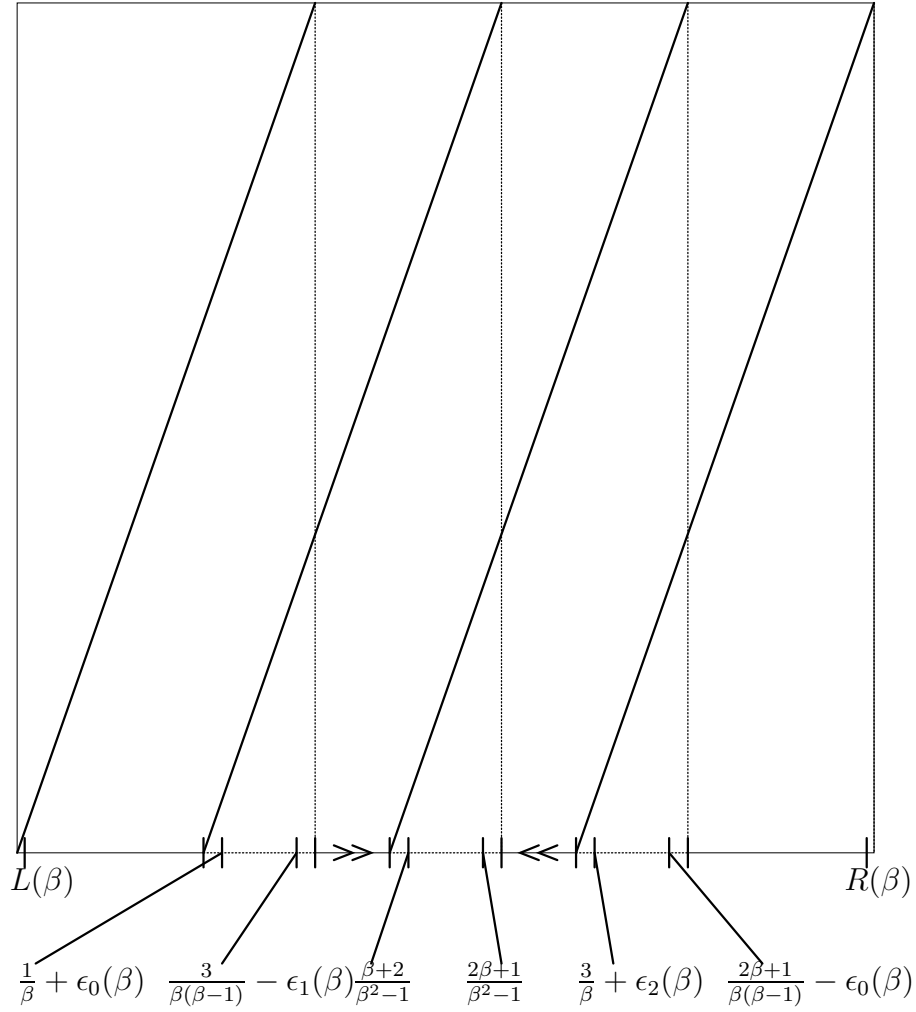
\begin{figure}[t]
\centering \unitlength=0.49mm
\begin{picture}(230,230)(0,-40)
\thinlines
\path(50,0)(0,0)(0,230)(230,230)(230,0)(180,0)
\dottedline(80,230)(80,0)
\dottedline(130,230)(130,0)
\dottedline(180,230)(180,0)
\dottedline(230,230)(230,0)
\dottedline(50,0)(80,0)
\dottedline(100,0)(130,0)
\dottedline(150,0)(180,0)
\path(80,0)(100,0)
\path(130,0)(150,0)
\thicklines
\path(0,0)(80,230)
\path(50,0)(130,230)
\path(100,0)(180,230)
\path(150,0)(230,230)
\path(2,4)(2,-4)
\path(50,4)(50,-4)
\path(80,4)(80,-4)
\path(100,4)(100,-4)
\path(130,4)(130,-4)
\path(150,4)(150,-4)
\path(180,4)(180,-4)
\path(228,4)(228,-4)
\path(105,4)(105,-4)
\path(125,4)(125,-4)
\path(55,4)(55,-4)
\path(75,4)(75,-4)
\path(155,4)(155,-4)
\path(175,4)(175,-4)

\path(55,0)(5,-30)
\path(75,0)(55,-30)
\path(105,0)(95,-30)
\path(125,0)(130,-30)
\path(155,0)(160,-30)
\path(175,0)(200,-30)

\path(85,2)(90,0)(85,-2)
\path(90,2)(95,0)(90,-2)
\path(140,2)(135,0)(140,-2)
\path(145,2)(140,0)(145,-2)

\put(0,-37){$\frac{1}{\beta}+\epsilon_{0}(\beta)$}
\put(40,-37){$\frac{3}{\beta(\beta-1)}-\epsilon_{1}(\beta)$}
\put(90,-37){$\frac{\beta+2}{\beta^{2}-1}$}
\put(120,-37){$\frac{2\beta+1}{\beta^{2}-1}$}
\put(150,-37){$\frac{3}{\beta}+\epsilon_{2}(\beta)$}
\put(190,-37){$\frac{2\beta+1}{\beta(\beta-1)}-\epsilon_{0}(\beta)$}

\put(-2,-10){$L(\beta)$}
\put(222,-10){$R(\beta)$}

\end{picture}
\caption{The interval $\mathcal{I}_{\beta}$ in the case where $m=3$ and $\beta\in[\frac{5}{2},1+\sqrt{3})$.}
    \label{fig4}
\end{figure}
\begin{prop}
Let $\beta\in[\frac{2k+3}{2},\frac{k+1+\sqrt{k^{2}+6k+5}}{2}).$ There exists $n(\beta)\in \mathbb{N}$ such that, for each $x\in \mathcal{I}_{\beta}$ there exists two elements $a,b\in \Omega_{\beta,m,n(\beta)}(x)$ such that $a(x)\in \mathcal{I}_{\beta}$ and $b(x)\in \mathcal{I}_{\beta}.$
\end{prop}
\begin{proof}
Without loss of generality we may assume that $\epsilon_{0}(\beta)$ is sufficiently small such that $\mathcal{I}_{\beta}$ contains the switch region. By Lemma \ref{switch lemma} there exists a sequence of maps $a$ that map $x$ into the switch region. As the endpoints of $\mathcal{I}_{\beta}$ are bounded away from the endpoints of $I_{\beta,m}$ we can bound the length of $a$ above by some $n_{s}(\beta)\in\mathbb{N}$. Moreover, by Lemma \ref{smaller switch 2} we may assume that $a(x)\in[\frac{1}{\beta}+\epsilon_{0}(\beta),\frac{(m-1)\beta+1}{\beta(\beta-1)}-\epsilon_{0}(\beta)].$ As in the even case it is useful to treat $[\frac{1}{\beta}+\epsilon_{0}(\beta),\frac{(m-1)\beta+1}{\beta(\beta-1)}-\epsilon_{0}(\beta)]$ as the union of subintervals. We observe that

\begin{align*}
\Big[\frac{1}{\beta}+\epsilon_{0}(\beta),\frac{(m-1)\beta+1}{\beta(\beta-1)}-\epsilon_{0}(\beta)\Big]=&\Big[\frac{1}{\beta}+\epsilon_{0}(\beta),\frac{m}{\beta(\beta-1)}-\epsilon_{1}(\beta)\Big]\\
\bigcup &\Big[\frac{m}{\beta}+\epsilon_{m-1}(\beta),\frac{(m-1)\beta+1}{\beta(\beta-1)}-\epsilon_{0}(\beta)\Big]\\
\bigcup &\Big[\frac{(k-1)\beta+m-(k-1)}{\beta(\beta-1)}-\epsilon_{k}(\beta),\frac{k+2}{\beta}+\epsilon_{k+1}(\beta)\Big]\\
\bigcup_{i=2}^{k}&\Big[\frac{i}{\beta}+\epsilon_{i-1}^{*}(\beta),\frac{(i-1)\beta+m-(i-1)}{\beta(\beta-1)}-\epsilon_{i}(\beta)\Big]\\
\bigcup_{i=k+2}^{m-1}&\Big[\frac{i}{\beta}+\epsilon_{i-1}(\beta),\frac{(i-1)\beta+m-(i-1)}{\beta(\beta-1)}-\epsilon_{i}^{*}(\beta)\Big]\\
\textrm{ }\bigcup_{i=1}^{k-1}&\Big[\frac{(i-1)\beta+m-(i-1)}{\beta(\beta-1)}-\epsilon_{i}(\beta),\frac{i+1}{\beta}+\epsilon_{i}^{*}(\beta)\Big]\\
\bigcup_{i=k+2}^{m-1}&\Big[\frac{(i-1)\beta+m-(i-1)}{\beta(\beta-1)}-\epsilon_{i}^{*}(\beta),\frac{i+1}{\beta}+\epsilon_{i}(\beta)\Big].
\end{align*}
Without loss of generality we may assume that $\epsilon_{0}(\beta),\epsilon_{i}(\beta),\epsilon_{i}^{*}(\beta)$ are all sufficiently small such that each of the above intervals in our union are well defined and nontrivial. We now proceed via a case analysis.

\begin{itemize}
	\item If $a(x)\in[\frac{1}{\beta}+\epsilon_{0}(\beta),\frac{m}{\beta(\beta-1)}-\epsilon_{1}(\beta)]$ then $T_{\beta,0}(a(x))\in\mathcal{I}_{\beta}$ and $T_{\beta,1}(a(x))\in\mathcal{I}_{\beta}.$ 
	\item If $a(x)\in[\frac{m}{\beta}+\epsilon_{m-1}(\beta),\frac{(m-1)\beta+1}{\beta(\beta-1)}-\epsilon_{0}(\beta)]$ then $T_{\beta,m-1}(a(x))\in\mathcal{I}_{\beta}$ and $T_{\beta,m}(a(x))\in\mathcal{I}_{\beta}.$
	\item Suppose $a(x)\in[\frac{(k-1)\beta+m-(k-1)}{\beta(\beta-1)}-\epsilon_{k}(\beta),\frac{k+2}{\beta}+\epsilon_{k+1}(\beta)].$ If $a(x)\in[\frac{k\beta+k+1}{\beta^{2}-1},\frac{(k+1)\beta+k}{\beta^{2}-1}]$ then $T_{\beta,k}(a(x))\in\mathcal{I}_{\beta}$ and $T_{\beta,k+1}(a(x))\in\mathcal{I}_{\beta}.$ If $a(x)\in[\frac{(k-1)\beta+m-(k-1)}{\beta(\beta-1)}-\epsilon_{k}(\beta),\frac{k\beta+k+1}{\beta^{2}-1}]$ then we are a bounded distance away from the fixed point of the map $T_{\beta,k},$ by Lemma \ref{Map properties} we know that $T_{\beta,k}$ scales the distance between $a(x)$ and the fixed point of $T_{\beta,k}$ by a factor $\beta,$ therefore we can bound the number of maps required to map $a(x)$ into $[\frac{k\beta+k+1}{\beta^{2}-1},\frac{(k+1)\beta+k}{\beta^{2}-1}]$. By a similar argument, if $a(x)\in[\frac{(k+1)\beta+k}{\beta^{2}-1},\frac{k+2}{\beta}+\epsilon_{k+1}(\beta)]$ we can bound the number of maps required to map $a(x)$ into $[\frac{k\beta+k+1}{\beta^{2}-1},\frac{(k+1)\beta+k}{\beta^{2}-1}]$. By the above we can assert that when $a(x)\in[\frac{(k-1)\beta+m-(k-1)}{\beta(\beta-1)}-\epsilon_{k}(\beta),\frac{k+2}{\beta}+\epsilon_{k+1}(\beta)]$ there exists two distinct sequences of maps whose length we can bound above by some $n_{c}(\beta)\in\mathbb{N}$ that map $a(x)$ into $\mathcal{I}_{\beta}.$
	\item If $a(x)\in[\frac{i}{\beta}+\epsilon_{i-1}^{*}(\beta),\frac{(i-1)\beta+m-(i-1)}{\beta(\beta-1)}-\epsilon_{i}(\beta)]$ for some $i\in\{2,\ldots,k-1\}$ then $T_{\beta,i-1}(a(x))\in\mathcal{I}_{\beta}$ and $T_{\beta,i}(a(x))\in\mathcal{I}_{\beta}.$
	\item If $a(x)\in[\frac{i}{\beta}+\epsilon_{i}(\beta),\frac{(i-1)\beta+m-(i-1)}{\beta(\beta-1)}-\epsilon_{i}^{*}(\beta)]$ for some $i\in\{k+2,\ldots,m-1\}$ then $T_{\beta,i-1}(a(x))\in\mathcal{I}_{\beta}$ and $T_{\beta,i}(a(x))\in\mathcal{I}_{\beta}.$
	\item If $a(x)\in[\frac{(i-1)\beta+m-(i-1)}{\beta(\beta-1)}-\epsilon_{i}(\beta),\frac{i+1}{\beta}+\epsilon_{i}^{*}(\beta)]$ for some $i\in\{1,\ldots,k-1\}$ then $a(x)$ is a bounded distance away from the fixed point of the map $T_{\beta,i},$ by Lemma \ref{Map properties} we know that $T_{\beta,i}$ scales the distance between $a(x)$ and its fixed point by a factor $\beta,$ therefore we can bound the number of maps required to map $a(x)$ outside of the interval $[\frac{(i-1)\beta+m-(i-1)}{\beta(\beta-1)}-\epsilon_{i}(\beta),\frac{i+1}{\beta}+\epsilon_{i}^{*}(\beta)]$ by some $n_{i}(\beta)\in\mathbb{N}.$ If $a(x)$ has been mapped into an interval covered by one of the above cases we are done, if not it has to be mapped into another interval of the form $[\frac{(j-1)\beta+m-(j-1)}{\beta(\beta-1)}-\epsilon_{j}(\beta),\frac{j+1}{\beta}+\epsilon_{j}^{*}(\beta)].$ By Corollary \ref{Increasing cor} and Lemma \ref{centre switch} we know that $i<j\leq k+1$. Repeating the previous step as many times as is necessary we can ensure that within $\sum_{i=1}^{k-1}n_{i}(\beta)$ maps, $a(x)$ has to be mapped into an interval that was addressed in one of our previous cases.
	\item The case where $a(x)\in[\frac{(i-1)\beta+m-(i-1)}{\beta(\beta-1)}-\epsilon_{i}^{*}(\beta),\frac{i+1}{\beta}+\epsilon_{i}(\beta)]$ for some $i\in\{k+2\ldots,m-1\}$ is analogous to the case where $a(x)\in[\frac{(i-1)\beta+m-(i-1)}{\beta(\beta-1)}-\epsilon_{i}(\beta),\frac{i+1}{\beta}+\epsilon_{i}^{*}(\beta)]$ for some $i\in\{1,\ldots,k-1\}.$
\end{itemize}

We've shown that for any $x\in \mathcal{I}_{\beta}$ there exists $n(x)\in\mathbb{N}$ such that, two distinct elements of $\Omega_{\beta,m,n(x)}(x)$ map $x$ into $\mathcal{I}_{\beta},$ moreover $n(x)\leq n_{s}(\beta)+n_{c}(\beta)+\sum_{i=1}^{k-1}n_{i}(\beta)$ . We take $n(\beta)$ to equal $n_{s}(\beta)+n_{c}(\beta)+\sum_{i=1}^{k-1}n_{i}(\beta)$. If $n(x)<n(\beta)$ then as in the even case we concatenate our image of $x$ by an arbitrary sequence of maps of length $n(\beta)-n(x)$ that map $x$ into $\mathcal{I}_{\beta},$ this ensures our sequences of maps are of length $n(\beta).$  

\end{proof}

Repeating the analysis given in the case where $m$ is even we can conclude Theorem \ref{Dimension theorem} in the case where $m$ is odd.

\section{Open questions and a table of values for $\mathcal{G}(m),$ $\beta_{f}(m)$ and $\beta_{c}(m)$}
We conclude with a few open questions and a table of values for $\mathcal{G}(m),$ $\beta_{f}(m)$ and $\beta_{c}(m)$. 

\begin{itemize}
	\item In \cite{AllClaSid} the authors study the order in which periodic orbits appear in the set of uniqueness. When $m=1$ they show that as $\beta\nearrow 2$ the order in which periodic orbits appear in the set of uniqueness is intimately related to the classical Sharkovskii ordering. It is natural to ask whether a similar result holds in our general case.
	\item In \cite{SidVer} it is shown that when $m=1$ and $\beta=\frac{1+\sqrt{5}}{2}$ the set of numbers: $x=\frac{(1+\sqrt{5})n}{2}(\textrm{mod } 1)$ for some $n\in\mathbb{N}$ have countably many $\beta$-expansions, while the other elements of $(0,\frac{1}{\beta-1})$ have uncountably many $\beta$-expansions. Does an analogue of this statement hold in the case of general $m$?
	\item Let $p_{1},\ldots,p_{k}$ be points in $\mathbb{R}^{d}$ such that the polyhedra $\Pi$ with these vertices is convex. Let $\{f_{i}\}_{i=1}^{k}$ be the one parameter family of maps given by $$f_{i}(x)=\lambda x +(1-\lambda)p_{i},$$ where $\lambda\in(0,1)$ is our parameter. As is well know there exists a unique $S_{\lambda}$ such that $S_{\lambda}=\cup_{i=1}^{k}f_{i}(S_{\lambda})$. We say that $(\epsilon_{i})_{i=1}^{\infty}\in\{1,\ldots,k\}^{\mathbb{N}}$ is an address for $x\in S_{\lambda}$ if $\lim_{n\to\infty} (f_{\epsilon_{n}}\circ\ldots \circ f_{\epsilon_{1}})(\textbf{0})=x.$ We ask whether an analogue of the golden ratio exists in this case, i.e, does there exists $\lambda^{*}$ such that for $\lambda\in(\lambda^{*},1)$ every $x\in S_{\lambda}\setminus\{p_{1},\ldots,p_{k}\}$ has uncountably many addresses, but for $\lambda\in(0,\lambda^{*})$ there exists $x\in S_{\lambda}\setminus\{p_{1},\ldots,p_{k}\}$ with a unique address. In \cite{SidOverlaps} the author shows that an analogue of the golden ratio exists in the case when $d=2$ and $k=3.$
\end{itemize}
\begin{table}[ht]
\caption{Table of values for $\mathcal{G}(m),$ $\beta_{f}(m)$ and $\beta_{c}(m)$} 
\centering  
\begin{tabular}{c c c c} 
\hline\hline                        
$m$ & $\mathcal{G}(m)$& $\beta_{f}(m)$ &$\beta_{c}(m)$\\ [0.5ex] 
\hline                  
1 & $\frac{1+\sqrt{5}}{2}\approx 1.61803\ldots$& $1.75488\ldots$ & $1.78723\ldots$\\ 
2 & 2 & $1+\sqrt{2}=2.41421\ldots$ & $2.47098\ldots$\\
3 & $1+\sqrt{3}\approx 2.73205\ldots$ & 2.89329\ldots &  2.90330\ldots\\
4 & $3$& $\frac{3+\sqrt{17}}{2}=3.56155\ldots$  & $3.66607\ldots$\\
5 & $\frac{3+\sqrt{21}}{2}\approx 3.79129\ldots$& $3.93947$ & $3.94583\ldots$\\
6 & $4$ & $2+\frac{\sqrt{28}}{2}=4.64575\ldots$ &  $4.75180\ldots$ \\ 
7 & $2+2\sqrt{2}\approx 4.82843\ldots $& 4.96095\ldots & $4.96496\ldots$\\
8 & $5$ &  $\frac{5+\sqrt{41}}{2}=5.70156\ldots$ & $5.80171\ldots$\\
9 & $\frac{5+\sqrt{45}}{2}\approx 5.85410\ldots $& $5.97273\ldots$ & $5.97537\ldots$\\
10& $6 $& $3+\sqrt{14}=6.74166\ldots$ & $6.83469\ldots$\\[1ex]      

\hline 
\end{tabular}
\label{table 1} 
\end{table}

\noindent \textbf{Acknowledgements} The author would like to thank Nikita Sidorov for much support and Rafael Alcaraz Barrera for his useful remarks.

\end{document}